\documentclass[10pt]{article}
\usepackage{blindtext}
\usepackage[a4paper, total={7in, 9in}]{geometry}

\makeatletter

\@addtoreset{equation}{section}
\makeatother

\usepackage[cp1251]{inputenc}
\usepackage[OT1]{fontenc}
\usepackage[english]{babel}
\usepackage[tbtags]{amsmath}
\usepackage{amsfonts, amsmath, amsthm, amssymb, graphicx, comment}
\usepackage{cite}

\graphicspath{{figures/}}

\newcommand{\figout}[1]
{#1}

\newtheorem{lemma}{Lemma}
\newtheorem{theorem}{Theorem}

\newtheorem{corollary}{Corollary}
\newtheorem{proposition}{Proposition}
\newtheorem{remark}{Remark}

\theoremstyle{definition}
\newtheorem{example}{Example}

\def\ds{\displaystyle}

\def\R{{\mathbb R}}

\def\SS{{\mathbb S}}
\def\H{{\mathbb H}}

\def\A{\mathcal{A}}
\def\U{\mathcal{U}}

\def\a{\alpha}
\def\b{\beta}
\def\g{\gamma}

\def\f{\varphi}
\def\d{\delta}
\def\D{\Delta}
\def\s{\sigma}
\def\eps{\varepsilon}
\def\p{\psi}
\def\om{\omega}

\def\gg{\mathfrak{g}}
\def\sh{\mathfrak{sh}}
\def\sl{\mathfrak{sl}}
\def\so{\mathfrak{so}}

\def\bq{\bar{q}}
\def\bx{\bar{x}}
\def\by{\bar{y}}

\def\lam{\lambda}

\def\tq{\widetilde{q}}

\def\td{\widetilde{d}}
\def\tx{\widetilde{x}}
\def\ty{\widetilde{y}}
\def\tX{\widetilde{X}}

\def\bg{\bar{g}}
\def\bq{\bar{q}}

\def\vH{\vec H}

\def\ss{\mathbf{s}}

\newcommand{\spann}{\mathop{\rm span}\nolimits}
\newcommand{\const}{\operatorname{const}\nolimits}
\newcommand{\VEC}{\operatorname{Vec}\nolimits}

\newcommand{\Id}{\operatorname{Id}\nolimits}

\newcommand{\Exp}{\operatorname{Exp}\nolimits}

\newcommand{\GL}{\operatorname{GL}\nolimits}

\newcommand{\cl}{\operatorname{cl}\nolimits}

\newcommand{\SL}{\operatorname{SL}\nolimits}
\newcommand{\SH}{\operatorname{SH}\nolimits}

\newcommand{\Lip}{\operatorname{Lip}\nolimits}
\newcommand{\intt}{\operatorname{int}\nolimits}
\newcommand{\sgn}{\operatorname{sgn}\nolimits}
\newcommand{\arsinh}{\operatorname{arsinh}\nolimits}
\newcommand{\arcosh}{\operatorname{arcosh}\nolimits}
\newcommand{\artanh}{\operatorname{artanh}\nolimits}
\newcommand{\Aff}{\operatorname{Aff}\nolimits_+(\R)}
\newcommand{\grad}{\operatorname{grad}\nolimits}

\def\Cinf{C^{\infty}(M)}

\def\then{\quad\Rightarrow\quad}

\def\Lo{Lo\-rent\-zi\-an }

\newcommand{\eq}[1]{$(\protect\ref{#1})$}
\newcommand{\be}[1]{\begin{equation}\label{#1}}
\newcommand{\ee}{\end{equation}}
\newcommand{\der}[2]{\frac{d \, #1}{d\, #2} }
\newcommand{\map}[3]{#1 \, : \, #2 \to #3}

\newcommand{\pder}[2]{\frac{\partial \, #1}{\partial \, #2} }
\newcommand{\restr}[2]{\left. #1 \right|_{#2}}

\newcommand{\onefiglabelsizen}[4]
{
\begin{figure}[htbp]
\begin{center}
\includegraphics[height=#4cm]{#1}
\\
\parbox[t]{0.7\textwidth}{\caption{#2}\label{#3}}
\end{center}
\end{figure}
}

\newcommand{\twofiglabelsizeh}[8]
{
\begin{figure}[htbp]
\includegraphics[height=#4cm]{#1}
\hfill
\includegraphics[height=#8cm]{#5}
\\
\parbox[t]{0.4\textwidth}{\caption{#2}\label{#3}}
\hfill
\parbox[t]{0.4\textwidth}{\caption{#6}\label{#7}}
\end{figure}
}

\title
{Lorentzian distance on the Lobachevsky plane\footnote{Work   supported by Russian Scientific Foundation, grant 22-11-00140, https://rscf.ru/project/22-11-00140/.
}}

\author
{Yu.L. Sachkov\\
  Program Systems Institute\\
  Russian Academy of Sciences\\
  Pereslavl-Zalessky, Russia\\[0.3cm]
yusachkov@gmail.com
}

\begin{document}

  \maketitle
	
	\begin{abstract}
	Left-invariant \Lo structures on the 2D solvable non-Abelian Lie group are studied. 
	Sectional curvature, attainable sets, \Lo length maximizers, distance, spheres, and infinitesimal isometries are described.
	\end{abstract}
	
	\tableofcontents

\section{Introduction}

Lorentzian geometry is the mathematical foundation of the theory of relativity \cite{wald, beem, muller}. It differs from the Riemannian one in that here information can propagate along curves with velocity vectors from some sharp cone.
Here, the natural problem is to find the  curves that maximize the length-type functional along admissible curves. Therefore, an important problem is to describe the Lorentzian length maximizers for all pairs of points where the second point is reachable from the first one along an admissible curve. As far as we know, this problem has been fully investigated only in the simplest case of a left-invariant Lorentzian structure in $\R^{n}$, for the Minkowski space $\R_1^{n}$ \cite{beem}.

This paper presents a description of Lorentzian length maximizers, distances and spheres for the next natural case --- for left-invariant Lorentzian structures on a unique connected simply connected non-Abelian two-dimensional Lie group. These results are obtained by methods of geometric control theory \cite{notes, intro}.
Curiously, in these problems, the Lorentzian length maximizers do not exist for some reachable pairs of points, and the \Lo distance can be infinite at some points. In these problems, all extremal trajectories (satisfying the Pontryagin maximum principle) are optimal, that is, there are neither conjugate points nor cut points. Optimal trajectories are parametrized by elementary functions, as are spheres and distances.

This work has the following structure. In Sec. \ref{sec:lor_geom} we recall necessary basic definitions of \Lo geometry.
In Sec. \ref{sec:lob} we describe the group of proper affine mappings of the real line $\Aff$ which bears the left-invariant \Lo problems stated in Sec. \ref{sec:problems}.
We show in Sec. \ref{sec:curv} that these problems have constant curvature $K$, thus are locally isometric to model \Lo spaces of constant curvature (2D Minkowski space for $K = 0$, de Sitter space for $K > 0$, anti-de Sitter space for $K < 0$).  
In Sec. \ref{sec:att} we describe positive and negative time attainable sets of the corresponding control systems.
Section \ref{sec:exist} is devoted to the study of existence of \Lo length maximizers.
In Sec. \ref{sec:extr} we apply the Pontryagin maximum principle to the problems studied and parametrize  geodesics.
In Sec. \ref{sec:max} we prove that in fact all geodesics are optimal, and construct explicitly optimal synthesis.
On the basis of these results in Sec. \ref{sec:dist} we describe \Lo distance and spheres.
In Sec. \ref{sec:isom} we describe Lie algebras of infinitesimal isometries (Killing vector fields) 
and the connected component of identity of the Lie groups of isometries
for the problems considered.
Moreover, in the case $K=0$ we construct explicitly an isometric embedding of $\Aff$ into a half-plane of the 2D Minkowski space.
Finally, in Sec. \ref{sec:ex} we specialize the results obtained to three model problems $P_1$, $P_2$, $P_3$.

\section{Lorentzian geometry}\label{sec:lor_geom}

A \Lo structure on a smooth manifold $M$ is a  nondegenerate quadratic form $g$ of index 1. \Lo geometry attempts to transfer the rich theory of Riemannian geometry  (in which the quadratic form $g$ is positive definite) to the case of Lorentzian metric $g$. 

Let us recall some basic definitions of \Lo geometry \cite{beem, muller}.  
A vector $v \in T_qM$, $q \in M$, is called timelike if $g(v)<0$,
spacelike if $g(v)>0$ or $v = 0$,
lightlike (or null) if $g(v)=0$ and $v \neq 0$, and 
nonspacelike if $g(v)\leq 0$.
A Lipschitzian curve in $M$ is called timelike if it has timelike velocity vector a.e.; spacelike, lightlike and nonspacelike curves are defined similarly.

A time orientation $X_0$ is an arbitrary timelike vector field in $M$. A nonspacelike vector $v \in T_qM$ is future directed if $g(v, X_0(q))<0$, and past directed if $g(v, X_0(q))>0$. 

A future directed timelike curve $q(t)$, $t \in [0, t_1]$, is called arclength paramet\-ri\-zed if $g(\dot q(t), \dot q(t)) \equiv - 1$. Any future directed timelike curve can be parametrized by arclength, similarly to Riemannian geometry.

The \Lo length of a nonspacelike curve $\g \in \Lip([0, t_1], M)$ is 
$\ds
l(\g) = \int_0^{t_1} |g(\dot \g, \dot \g)|^{1/2} dt.
$
For points $q_0, q_1 \in M$ denote by $\Omega_{q_0q_1}$ the set of all future directed nonspacelike curves in $M$ that connect $q_0$ to $q_1$. In the case $\Omega_{q_0q_1} \neq \emptyset$ define the \Lo distance (time separation function) from the point $q_0$ to the point $q_1$ as
\be{d}
d(q_0, q_1) = \sup \{l(\g) \mid \g \in \Omega_{q_0q_1}\}.
\ee
And if 
 $\Omega_{q_0q_1} = \emptyset$, then by definition $d(q_0, q_1)= 0$.
A future directed nonspacelike curve $\g$ is called a \Lo length maximizer if it realizes the supremum in \eq{d} between its endpoints $\g(0) = q_0$, $\g(t_1) = q_1$.

The causal future of a point $q_0 \in M$ is the set $J^+(q_0)$ of points $q_1 \in M$ for which there exists a future directed nonspacelike curve $\g$ that connects $q_0$ and $q_1$. 
The causal past $J^-(q_0)$ is defined analogously in terms of past directed nonspacelike curves.
The chronological future $I^+(q_0)$ and chronological past $I^-(q_0)$ of a point $q_0 \in M$ are defined similarly via future directed and past directed timelike curves $\g$. 

Let $q_0 \in M$, $q_1 \in J^+(q_0)$. The search for \Lo length maximizers that connect $q_0$ with $q_1$ reduces to the search for future directed nonspacelike curves $\g$ that solve the problem
\be{lmax}
l(\g) \to \max, \qquad \g(0) = q_0, \quad \g(t_1) = q_1.
\ee

A set of vector fields $X_1, \dots, X_n \in \VEC(M)$, $n = \dim M$, is an orthonormal frame for a \Lo structure $g$ if for all $q \in M$
\begin{align*}
&g_q(X_1, X_1) = -1, \qquad g_q(X_i, X_i) = 1, \quad i = 2, \dots, n, \\
&g_q(X_i, X_j) = 0, \quad i \neq j. 
\end{align*}
Assume that time orientation is defined by a timelike vector field $X \in \VEC(M)$ for which $g(X, X_1) < 0$ (e.g., $X = X_1$). Then the   \Lo problem for the \Lo structure with the orthonormal frame $X_1, \dots, X_n$ is stated as the following optimal control problem:
\begin{align*}
&\dot q = \sum_{i=1}^n u_i X_i(q), \qquad q \in M, \\
&u \in U = \left\{(u_1, \dots, u_n) \in \R^n \mid u_1 \geq \sqrt{ u_2^2 + \dots + u_n^2}\right\},\\
&q(0) = q_0, \qquad q(t_1) = q_1, \\
&l(q(\cdot)) = \int_0^{t_1} \sqrt{u_1^2 - u_2^2 -  \dots - u_n^2} \, dt  \to \max.
\end{align*}

\begin{remark}
The \Lo length is preserved under monotone Lipschitzian time reparametrizations $t(s)$, $s \in [0, s_1]$. Thus if $q(t)$, $t \in [0, t_1]$, is a \Lo length maximizer, then so is any its reparametrization $q(t(s))$, $s \in [0, s_1]$. 

In this paper we choose primarily the following parametrization of trajectories: the arclength parametrization ($u_1^2 - u_2^2 - \cdots - u_n^2 \equiv 1$) for timelike trajectories, and the parametrization with $u_1(t) \equiv 1$ for future directed lightlike trajectories. Another reasonable choice is to set $u_1(t) \equiv 1$ for all future directed nonspacelike trajectories.
\end{remark}

\begin{remark}
In Lorentzian geometry, only nonspacelike curves have a physical meaning since according to the Relativity Theory information cannot move with a speed greater than the speed of light
{\em \cite{wald, beem, muller}}. By this reason, in \Lo geometry typically only nonspacelike curves are studied.

Geometrically, spacelike curves may well be considered. For $2$-dimensional Lorentzian manifolds there is not much geometric difference between timelike and spacelike curves since the first ones are obtained from the second ones by a change of 
\Lo form $g \mapsto -g$, or, equivalently, by a change of
controls $(u_1, u_2) \mapsto (u_2, u_1)$.
Although, for Lorentzian manifolds of dimension greater than $2$ the spacelike cone is nonconvex, so the optimization problem of finding the longest spacelike curve is not well-defined (optimal trajectories do not exist).

Notice also that
curves $q(\cdot)$ of variable causality ($\sgn g(\dot q) \neq \const$) cannot be optimal: it is easy to show that the causal character of extremal trajectories is preserved.
\end{remark}

\begin{remark}
The  \Lo distance is defined   by maximization \eq{d}, not by minimization as in Riemannian geometry. In Lorentzian geometry, the distance means physically the space-time interval between events in a space-time {\em \cite{wald, beem, muller}}.
On the other hand, the minimum of \Lo length is always zero (by virtue of lightlike trajectories), so the minimization problem here is not  interesting. 

Notice also that the \Lo distance $d$ is not a distance (metric) in the sense of metric spaces since $d$ is not symmetric and satisfies the {\em reverse} triangle inequality.
\end{remark}

\begin{example}
The simplest example of Lorentzian geometry is the Minkowski  space \cite{beem}. In the 2D case it is defined as 
$
\R^2_1 = \R^2_{xy}$, $g = -dx^2 + dy^2$.
The \Lo length maximizers are straight line segments along which $g \leq 0$, the \Lo distance is 
$$d((x_0, y_0), (x_1, y_1)) = 
\begin{cases}
\sqrt{(x_1-x_0)^2 - (y_1-y_0)^2} &\text{ for $(x_1-x_0)^2 - (y_1-y_0)^2 > 0$}, \\
0 &\text{ for $(x_1-x_0)^2 - (y_1-y_0)^2 \leq 0$},
\end{cases}
$$
and positive radius \Lo spheres are arcs of hyperbolas with asymptotes parallel to lightlike curves $x = \pm y$. See Fig. \ref{fig:minkow}.
\end{example}

This example has the following generalizations and variations, see \cite{Wolf}, Sec. 5.2.
Let $\R^n_s$, $0 \leq s \leq n$, denote the vector space $\R^n$ with the quadratic form
$
g^n_s = - \sum_{i=1}^s dx_i^2 + \sum_{j=s+1}^n dx_j^2.
$

\begin{example}
Let $n \geq 2$.
The Minkowski space $\R^n_1$ is a \Lo manifold with the \Lo form $g^n_1$.
It has constant curvature $K= 0$ (\cite{Wolf}, Th. 2.4.3).
\end{example}

\begin{example}
Let $n \geq 2$, and let $r > 0$.
The de Sitter space is the \Lo manifold
$$
\SS^n_1 = \left\{x = (x_1, \dots, x_{n+1}) \in \R^{n+1}_1 \mid - x_1^2 + x_2^2 + \dots + x_{n+1}^2 = r^2\right\}
$$ 
 with the \Lo form $g = \restr{g^{n+1}_1}{\SS^n_1}$.
The space 
$\SS^n_1$ has constant curvature $K= \frac{1}{r^2}$ (\cite{Wolf}, Th. 2.4.4).

Consider the \Lo manifold
$$
\H^n_1 = \left\{x = (x_1, \dots, x_{n+1}) \in \R^{n+1}_2 \mid - x_1^2 - x_2^2 + x_3^2 + \dots + x_{n+1}^2 = -r^2\right\}
$$ 
 with the \Lo form $g = \restr{g^{n+1}_2}{\H^n_1}$.
The universal covering $\widetilde{\H^n_1}$ of $\H^n_1$ is called anti-de Sitter space.
The spaces ${\H^n_1}$ and
$\widetilde{\H^n_1}$ have constant curvature $K= -\frac{1}{r^2}$ (\cite{Wolf}, Th. 2.4.4).
\end{example}

Let $M_j$ be a \Lo manifold with \Lo distance $d_j$, $j = 1, 2$. A mapping $\map{i}{M_1}{M_2}$ is called an isometry if $d_1(q, p) = d_2(i(q), i(p))$ for all $q, p \in M_1$. 

\begin{example}
The group of isometries of the Minkowski plane $\R^2_1$ is generated by  translations, hyperbolic rotations $e^{t X}$, $X = y \pder{}{x} + x \pder{}{y}$, and reflections $(x, y) \mapsto (x, -y)$.
\end{example}

\section{Lobachevsky plane }\label{sec:lob}

Proper affine functions on the line  are mappings of the form
\be{affine}
a \mapsto y \cdot a + x, \qquad a \in \R, \quad y > 0, \quad x \in \R.
\ee
Consider the group of such functions  
$\ds
G = \Aff = \{(x, y) \in \R^2 \mid y > 0\}
$
with the group product induced by composition of functions \eq{affine}:
$$
(x_2, y_2) \cdot (x_1, y_1) = (x_2 + y_2 x_1, y_2 y_1), \qquad (x_i, y_i) \in G
$$
and the identity element $\Id = (0, 1) \in G$. This group is a semi-direct product $\Aff = \R_+ \rtimes \R$.

$G$ is a two-dimensional Lie group, connected and simply connected. The vector fields $X_1 = y \pder{}{x}$, $X_2 = y \pder{}{y}$ form a left-invariant frame on $G$, thus the Lie algebra of $G$ is $\gg = \spann(X_1, X_2)$. In view of the Lie bracket $[X_2, X_1] = X_1$, $\gg$ and $G$ are solvable and non-Abelian. In fact, $\gg$ is a unique solvable non-Abelian two-dimensional Lie algebra \cite{jacobson}.

One-parameter subgroups in $G$ are  rays (or straight lines if $u_2 = 0$)
$$
u_1 (y-1) = u_2 x, \qquad (u_1, u_2) \neq (0, 0), \quad (x, y) \in G,
$$
with the parametrization
\begin{align}
&x = \frac{u_1}{u_2}(e^{u_2 t}-1), \quad y= e^{u_2 t}, \qquad u_2 \neq 0, \label{1par1}\\
&x = {u_1}t, \quad y=  1, \qquad u_2 = 0,  \label{1par2}
\end{align}
 see Fig. \ref{fig:1par}. Formulas \eq{1par1}, \eq{1par2} for $t=1$ describe the exponential mapping
\be{exp}
\map{\exp}{\gg}{G}, \qquad u_1 X_1 + u_2 X_2 \mapsto (x, y)(1).
\ee
Notice that left translations of one-parameter   subgroups in $G$ are also rays (or straight lines if $u_2 = 0$) since left translations in $G$ are compositions of homotheties with parallel translations in $\R^2_{x, y}$.

\figout{
\twofiglabelsizeh
{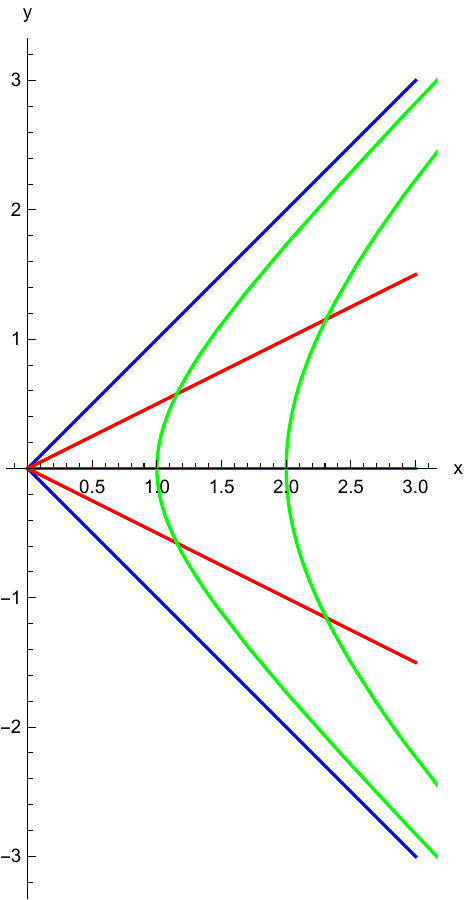}{2D Minkowski space}{fig:minkow}{8}
{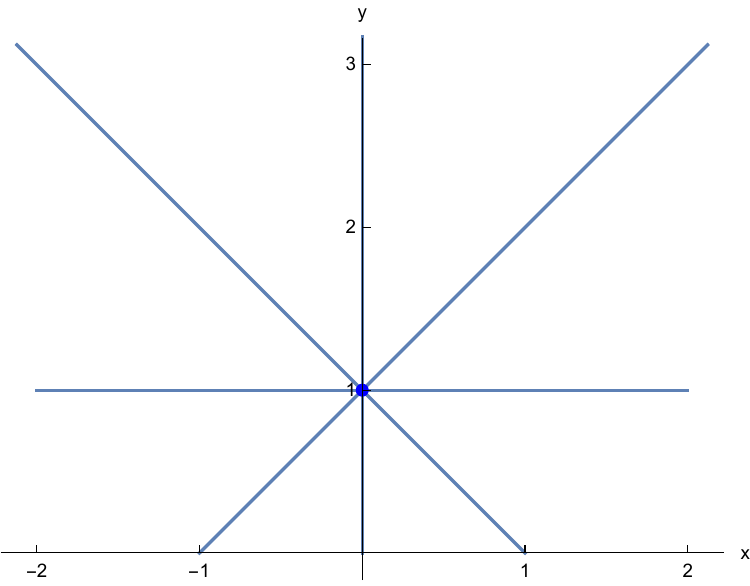}{One-parameter subgroups in~$G$}{fig:1par}{7}
}

\begin{remark}
Riemannian geometry on $\Aff$ with the orthonormal frame $X_1$, $X_2$ is the Lobachevsky (Gauss, Bolyai) non-Euclidean geometry (in Poincar\'e's model in the upper half-plane) {\em \cite{coxeter, stahl}}.
\end{remark}

\section{Left-invariant Lorentzian problems on the Lobachevsky plane}\label{sec:problems}
In this work we consider left-invariant \Lo problems on the Lie group $G = \Aff$. Such a problem is specified by an index 1 quadratic form $g$ on the Lie algebra $\gg$ and a timelike time orientation vector field $X_0 \in \gg$. 

A Lipschitzian curve $\map{q}{[0, t_1]}{G}$ is a \Lo length maximizer that connects the point $\Id$ to a point $q_1 \in G$ iff it is a solution to the following optimal control problem:
\begin{align}
&g(\dot q(t)) \leq 0, \qquad \bg(\dot q(t), X_0(q(t))<0, \label{pr11} \\
&q(0) = \Id, \qquad q(t_1) = q_1, \label{pr12} \\
&l = \int_0^{t_1} |g(\dot q(t))|^{1/2} dt \to \max, \label{pr13}
\end{align}
where $\bg$ is the bilinear form on $\gg$ corresponding to the quadratic form $g$.

Let us decompose a vector $\gg \ni v = u_1 X_1 + u_2 X_2$, then the \Lo form $g$ and the bilinear form $\bg$ are represented as $g(v) = g(u_1, u_2)$, $\bg(v^1, v^2) = \bg(v_1^1, v_2^2; v_1^2, v_2^2)$, where $v^i = v_1^iX_1 + v_2^i X_2$. Let $X_0 = v_1^0 X_1 + v_2^0X_2$, and denote the linear form $g_0(u_1, u_2) = \bg(v_1^0, v_2^0; u_1, u_2)$. Then the \Lo problem \eq{pr11}--\eq{pr13} reads as
\begin{align}
&\dot q(t) = u_1 X_1 + u_2 X_2, \qquad q \in G, \quad u = (u_1, u_2) \in U, \label{pr21}\\
&U = \{ u \in \R^2 \mid g(u) \leq 0, \ g_0(u) < 0\}, \label{pr22}\\
&q(0) = \Id, \qquad q(t_1) = q_1, \label{pr23} \\
&l = \int_0^{t_1} |g(u)|^{1/2} dt \to \max. \label{pr24}
\end{align}

The \Lo quadratic form can be decomposed as a sum of squares
\begin{align}
&g(u) = - (au_1+bu_2)^2 + (cu_1+du_2)^2, \label{gabcd} \\
&A = \left(
\begin{array}{cc} a & b \\ c & d \end{array}
\right) 
\in \GL(2, \R).\label{A}
\end{align}
Notice that the matrix $A$ in \eq{A} is not unique: it is determined up to the symmetries
$$
\eps_1 \ : \ (a, b, c, d) \mapsto (-a, -b, c, d), \qquad\qquad\qquad
\eps_2 \ : \ (a, b, c, d) \mapsto (a, b, -c, -d).
$$
The inequality $\restr{g_0}{U} = a u_1 + bu_2 < 0$ fixes signs of $a$ and $b$, thus killing the reflection $\eps_1$. If we further assume that $|A|>0$ in \eq{A}, then the signs of $c$ and $d$ become fixed, thus $\eps_2$ is killed. Summing up, we have the following.

\begin{lemma}
The space of left-invariant \Lo problems \eq{pr21}--\eq{pr24} is  parametrized by matrices 
$$A = \left(
\begin{array}{cc} a & b \\ c & d \end{array}
\right) \in \GL_+(2, \R) = \{ A \in \GL(2, \R) \mid |A|>0\}.
$$ 
\end{lemma}

Given a problem \eq{pr21}--\eq{pr24} determined by a matrix $A = \left(
\begin{array}{cc} a & b \\ c & d \end{array}
\right) \in \GL_+(2, \R)$,  introduce new controls
$$
v_1 = au_1 + bu_2, \qquad v_2 = cu_1+du_2, 
$$
or, equivalently,
$$
u_1 = \a v_1 + \b v_2, \qquad u_2 = \g v_1 + \d v_2, \qquad
\left(
\begin{array}{cc} \a & \b \\ \g & \d \end{array}
\right)
= A^{-1}.
$$
Introduce further the vector fields
$$
Y_1 = \a X_1 + \g X_2, \qquad Y_2 = \b X_1 + \d X_2.
$$
Then the problem \eq{pr21}--\eq{pr24} reads as 
\begin{align}
&\dot q = v_1 Y_1 + v_2 Y_2, \qquad q \in G,  \label{pr31}\\
&g = -v_1^2 +v_2^2 \leq 0, \qquad g_0 = - v_1  < 0, \label{pr32}\\
&q(0) = \Id, \qquad q(t_1) = q_1, \label{pr33} \\
&l = \int_0^{t_1} \sqrt{v_1^2 - v_2^2} dt \to \max. \label{pr34}
\end{align}

The \Lo form factorizes as
\begin{align*}
&g = l_1l_2, \qquad
l_1(u_1, u_2) = (c-a)u_1 + (d-b)u_2, \qquad
l_2(u_1, u_2) = (c+a)u_1+(d+b)u_2.
\end{align*}
Introduce the corresponding functions on $G$:
\begin{align*}
&\lam_1(x, y) = \grad l_1 \cdot
\left(\begin{array}{c}
x \\ y-1\end{array}\right) = (c-a)x + (d-b)(y-1), \\
&\lam_2(x, y) = \grad l_2 \cdot
\left(\begin{array}{c}
x \\ y-1\end{array}\right) =(c+a)x+(d+b)(y-1).
\end{align*}


\begin{remark}
By virtue of the change of variables
$
(u_1, u_2) \mapsto (-u_1, -u_2)$, $A \mapsto - A$, $t \mapsto - t$,
we can get
\be{a>0}
a\geq 0 \text{ or, equivalently, } \d \geq 0,
\ee
which we assume   in the sequel.
\end{remark}

\begin{example}\label{ex:P13}
As typical examples of Lorentzian problems \eqref{pr31}--\eqref{pr34}, we consider in Sec. \ref{sec:ex} the following
model  problems $P_i$, $i = 1, 2, 3$:
\begin{itemize}
\item[$P_1$:] $A = \left(
\begin{array}{cc} 1 & 0 \\ 0 & 1 \end{array}
\right)$,
$U = \{u=(u_1, u_2) \in \R^2 \mid -u_1^2+ u_2^2\leq 0, \ -u_1 \leq 0\}$, $g = -u_1^2+ u_2^2$, $g_0 = -u_1$,
\item[$P_2$:] 
$A = \left(
\begin{array}{cc} 0 & 1 \\ -1 & 0 \end{array}
\right)$,
$U = \{u=(u_1, u_2) \in \R^2 \mid -u_2^2+ u_1^2\leq 0, \ -u_2 \leq 0\}$, $g = -u_2^2+ u_1^2$, $g_0 = -u_2$,
\item[$P_3$:] 
$A = \left(
\begin{array}{cc} 1/2 & 1/2 \\ -1/2 & 1/2 \end{array}
\right)$,
$U = \{u=(u_1, u_2) \in \R^2 \mid u_1 \geq 0, \ u_2 \geq 0\}$, $g = -u_1 u_2$, $g_0 = -(u_1+u_2)/2$.
\end{itemize}
See the sets of control parameters $U$ for these problems resp. in Figs. \ref{fig:UP1}--\ref{fig:UP3}.

\figout{
\twofiglabelsizeh
{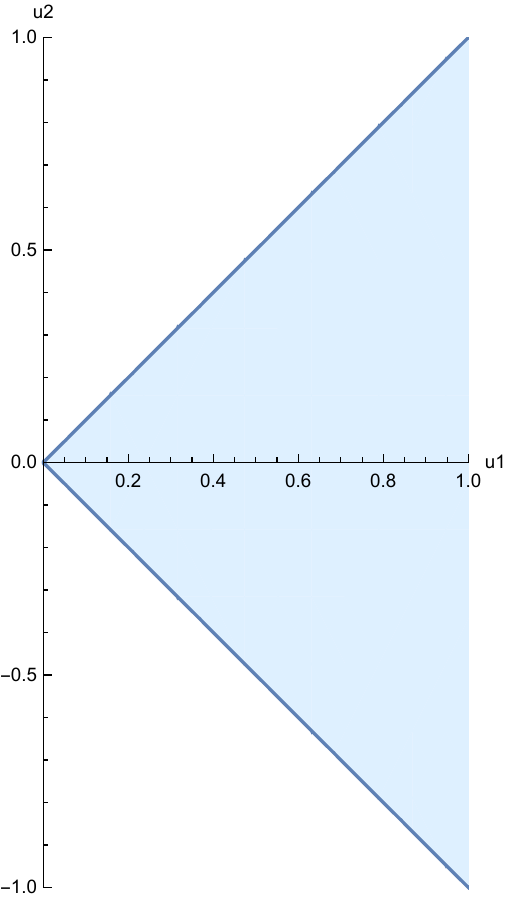}{The set $U$ for the problem $P_1$}{fig:UP1}{8}
{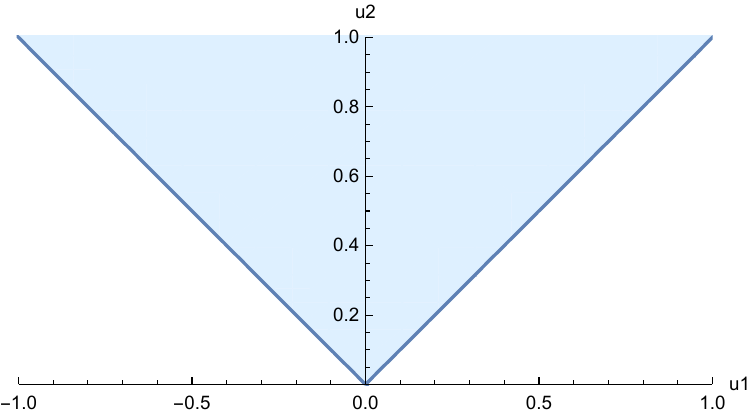}{The set $U$ for the problem $P_2$}{fig:UP2}{4}

\onefiglabelsizen
{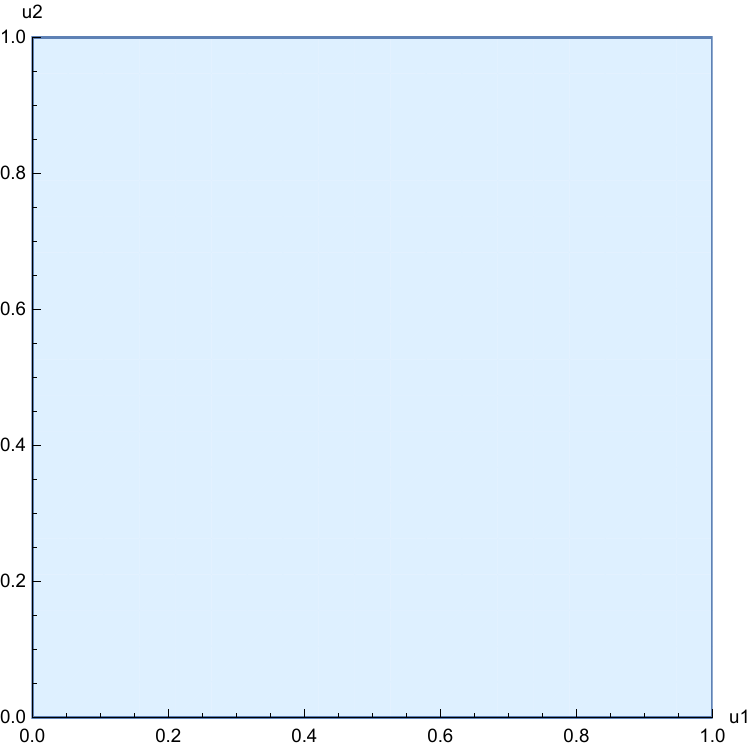}{The set $U$ for the problem $P_3$}{fig:UP3}{6}
}
\end{example}

We denote
$
J^+ = J^+(\Id)
$
and $d(q) = d(\Id, q)$, $q \in G$.

\section{Curvature}\label{sec:curv}
In this section we show that each left-invariant \Lo structure on the group $G = \Aff$ has constant sectional curvature $K$, thus it is locally isometric to the 2D Minkowski space (if $K = 0$), to a 2D de Sitter space (if $K > 0$), or to a 2D anti-de Sitter space (if $K < 0$).

\subsection{Levi-Civita connection
and sectional curvature of \Lo manifolds}
Here we recall some standard facts of \Lo (in fact, pseudo-Riemannian) geometry, following \cite{Neill, beem, Wolf}.

A connection $D$ on a smooth manifold $M$ is a mapping $\map{D}{(\VEC(M))^2}{\VEC(M)}$ such that 
\begin{itemize}
\item[(1)]
$D_VW$ is $\Cinf$-linear in $V$,
\item[(2)]
$D_VW$ is $\R$-linear in $W$,
\item[(3)]
$D_V(fW) = (Vf) W + s D_VW$ for $f \in \Cinf$.
\end{itemize}
The vector field 
$D_VW$ is called the covariant derivative of $W$ w.r.t. $V$ for the connection $D$.

\begin{theorem}[\cite{Neill}, Th. 11]
On a \Lo manifold $(M, g)$ there is a unique connection $D$ such that
\begin{itemize}
\item[$(4)$]
$[V, W] = D_VW-D_WV$, and
\item[$(5)$]
$Xg(V, W) = g(D_XV, W) + g(V, D_XW)$,
\end{itemize}
for all $X, V, W \in \VEC(M)$.
$D$ is called the Levi-Civita connection on $M$, and is characterized by the Koszul formula
$$
2g(D_VW, X) = Vg(W, X) + Wg(X, V) - Xg(V, W)
-g(V, [W, X]) + g(W,[X, V]) + g(X, [V, W]).
$$
\end{theorem}

Let $(M, g)$ be a \Lo manifold with Levi-Civita connection $D$. The mapping $\map{R}{(\Cinf)^3}{\Cinf}$ given by
$
R_{XY}Z = D_{[X, Y]}Z - [D_X, D_Y]Z
$
is called the Riemannian curvature tensor of $(M,g)$.

Let $q \in M$, and let $P$ be a 2D plane in $T_qM$. For vectors $v, w \in T_qM$, define
$
Q(v, w) = g(v, v) g(w, w) - (g(v, w))^2.
$
A plane $P$ is called nondegenerate if $Q(v, w) \neq 0$ for some (hence every) basis $v, w$ for $P$.

\begin{lemma}[\cite{Neill}, Lemma 39]
Let $P \subset T_qM$ be a nondegenerate plane. The number
\be{KqP}
K(q, P) = \frac{g(R_{vw}v, w)}{Q(v, w)}
\ee
is independent of the choice of basis $v, w$ in $P$, and is called the sectional curvature of the plane section $P$.
\end{lemma}

A \Lo manifold which has the same sectional curvature on all nondegenerate sections is said to have constant curvature.

\begin{theorem}[\cite{Wolf}, Theorem 2.4.1]
Let $(M, g)$ be a \Lo manifold of dimension $n \geq 2$, and let $K \in \R$. Then the following conditions are equivalent:
\begin{itemize}
\item[$(1)$]
$M$ has constant curvature $K$,
\item[$(2)$] 
for any $q \in M$ there exists a neighbourhood of $q$ isometric to an open subset of de Sitter space $\SS^n_1$ for $K > 0$, Minkowski space $\R^n_1$ for $K = 0$, anti-de Sitter space $\widetilde{\H^n_1}$ for $K < 0$.
\end{itemize}
\end{theorem}

\subsection{Sectional curvature of $\Aff$}
In this subsection we compute Levi-Civita connection  and sectional curvature of left-invariant \Lo structures on the group $G = \Aff$.

\begin{theorem}
Levi-Civita connection $D$ of a left-invariant \Lo structure $g$ on the group $G = \Aff$ is given as follows:
\begin{align*}
&D_{X_i}X_j = \mu_{ij}X_1 + \nu_{ij}X_2, \qquad i, j = 1, 2, \\
&(\mu_{11}, \nu_{11}) = - \frac{1}{|A|^2}(-g_{12}g_{11}, g_{11}^2),\qquad
(\mu_{12}, \nu_{12}) = - \frac{1}{|A|^2}(-g_{22}g_{11}, g_{12}g_{11}), \\
&(\mu_{21}, \nu_{21}) = - \frac{1}{|A|^2}(-g_{12}^2, g_{11}g_{12}),\qquad
(\mu_{22}, \nu_{22}) = - \frac{1}{|A|^2}(-g_{22}g_{12}, g_{12}^2), \\
&g_{11} = g(X_1) = c^2-a^2, \qquad g_{12} = g(X_1, X_2) = cd-ab, \qquad 
g_{22} = g(X_2) = d^2-b^2.
\end{align*}
\end{theorem}
\begin{proof}
Immediate computation via Koszul formula.
\end{proof}

\begin{theorem}\label{th:K}
A left-invariant \Lo structure $g$ on the group $G = \Aff$ has constant curvature $\ds K = \frac{g(X_1)}{|A|^2}$.
\end{theorem}
\begin{proof}
Immediate computation via formula \eq{KqP} for $P = T_qG$, $v = X_1(q)$, $w = X_2(q)$, $q \in G$.
\end{proof}

\begin{corollary}
A left-invariant \Lo structure $g$ on the group $G = \Aff$ is locally isomorphic to the $2D$ Minkowski space $\R^2_1$ (for $K= 0$), de Sitter space $\SS^2_1$ (for $K>0$), or  anti-de Sitter space $\widetilde{\H^2_1}$ (for $K<0$).
\end{corollary}

\begin{remark}
For the case $K = 0$
we construct an explicit isometry of the group $G$ to a half-plane of $\R^2_1$   in Th. $\ref{th:iso}$.
\end{remark}

\section{Attainable sets}\label{sec:att}


Denote the set of admissible velocities
$
\U = \{u_1 X_1 + u_2 X_2 \mid (u_1, u_2) \in U\} \subset \gg.
$

\begin{theorem}\label{th:J}
Let $q_0 \in G$, then 
\begin{align}
&J^+(q_0) = q_0 \exp(\U) = \{ q \in G \mid \lam_1(q) \leq \lam_1(q_0), \ \lam_2(q) \geq \lam_2(q_0)\}, \label{J+}\\
&J^-(q_0) = q_0 \exp(-\U) = \{ q \in G \mid \lam_1(q) \geq \lam_1(q_0), \ \lam_2(q) \leq \lam_2(q_0)\}, \label{J-}
\end{align}
where $\map{\exp}{\gg}{G}$ is the exponential mapping \eq{exp}
of the Lie group $G$.

Moreover, $I^+(q_0) = \{q_0\} \cup \intt J^+(q_0)$ and $I^-(q_0) = \{q_0\} \cup \intt J^-(q_0)$.
\end{theorem}
\begin{proof}
By left-invariance of the problem, we need to prove equalities \eq{J+}, \eq{J-} in the case $q_0 = \Id$ only. 

Let us show that 
\be{J+Id}
J^+ = \exp(\U) = \{ q \in G \mid \lam_1(q) \leq 0 \leq \lam_2(q)\}.
\ee
Future oriented nonspacelike one-parameter semigroups
$$
\{\exp(t(u_1, u_2)) \mid t \geq 0 \} = \{(x, y) \in G \mid  u_1(y-1) = u_2 x\}, \qquad l_1(u_1, u_2) \leq 0 \leq l_2(u_1, u_2), 
$$
fill the set $\exp(\U)$, thus $J^+ \supset \exp(\U)$. On the other hand,
admissible trajectories of the system \eq{pr31}, \eq{pr32} at the boundary of $\exp(\U)$ are tangent to $\partial \exp(\U)$ or are directed inside $\exp(\U)$. Thus $J^+ \subset \exp(\U)$, and equality \eq{J+Id} follows. 

A similar equality for $J^-(\Id)$ is proved analogously.
The expressions for $I^{\pm}(q_0)$ are straightforward.
\end{proof}

See the set $J^+$ for the problems $P_1$, $P_2$, $P_3$ in Figs. \ref{fig:att_set1}, \ref{fig:att_set2}, \ref{fig:att_set3} respectively.


\section{Existence of \Lo length maximizers}\label{sec:exist}

\subsection{Existence of  length maximizers for globally hyperbolic \Lo structures}
In order to study existence of \Lo length maximizers we need some facts from \Lo geometry \cite{beem}.

Let $M$ be a \Lo manifold. An open subset $O \subset M$ is called causally convex if the intersection of each nonspacelike curve with $O$ is connected. $M$ is called strongly causally convex in any point in $M$ has arbitrarily small causally convex neighbourhoods. Finally, a strongly causally convex \Lo manifold $M$ is called globally hyperbolic if 
\be{compact}
J^+(p) \cap J^-(q) \text{ is compact for any } p, q \in M.
\ee

\begin{theorem}[Th. 6.1 \cite{beem}]\label{th:Avez}
If a \Lo manifold $M$ is 
 globally hyperbolic, then any points $q_0 \in M$,  $q_1 \in J^+(q_0)$ can be connected by a \Lo length maximizer. 
\end{theorem}

\begin{theorem}\label{th:glob_hyp}
A \Lo structure $(g, X_0)$ on $\Aff$ is globally hyperbolic iff $K \geq 0$.
\end{theorem}
\begin{proof}
First, all left-invariant \Lo structures on $\Aff$ are strongly causally convex.
Indeed, $\dot x= u_1 y$ or
$\dot y = u_2 y$ preserves sign and is separated from zero for
 $(x, y) \in O$, $u_1^2+ u_2^2 \geq C>0$, $g(u) \leq 0$, $g_0(u) < 0$.

So we need to check condition \eq{compact} only.
It follows from Th. \ref{th:J} that for $K \geq 0$ the intersection in \eq{compact} is compact (it is either a parallelogram, a segment, or the empty set). The same theorem implies that for $K < 0$ there exist points $q \in G$ such that the intersection $J^+ \cap J^-(q)$ contains points from the absolute $\{y=0\}$ in its closure, thus this intersection is not compact.
\end{proof}

\begin{theorem}
Let $K \geq 0$. Then for any points $q_0 \in G$, $q_1 \in J^+(q_0)$ there exists a \Lo length maximizer from $q_0$ to $q_1$. 
\end{theorem}
\begin{proof}
Follows from Theorems \ref{th:Avez}, \ref{th:glob_hyp}.
\end{proof}

\subsection{Existence of  length maximizers in the case $K< 0$}
In this subsection we consider the remaining case $K< 0$.
Introduce the decomposition
\begin{align}
&J^+ = D \sqcup F \sqcup E, \label{DFE}\\
&D = \{q \in G \mid \lam_1(q) \leq 0 \leq \lam_2(q), \ \lam_3(q) > 0\}, \qquad
F = \{q\in G \mid \lam_3(q) = 0\}, \qquad
E = \{ q \in G \mid \lam_3(q) < 0\}, \nonumber\\
&\lam_3(q) = \lam_1(q) - \lam_1(B), \qquad B = \left( \frac{d+b}{c+a}, 0\right) \in \R^2, \nonumber
\end{align}
so that the lines $\{q \in \R^2 \mid \lam_2(q) = 0 \}$ and the absolute $\{y=0\}$ intersect at the point $B \in \R^2_{x, y} \setminus G$, see Fig. \ref{fig:att_set25} for the problem $P_1$. 

\figout{
\twofiglabelsizeh
{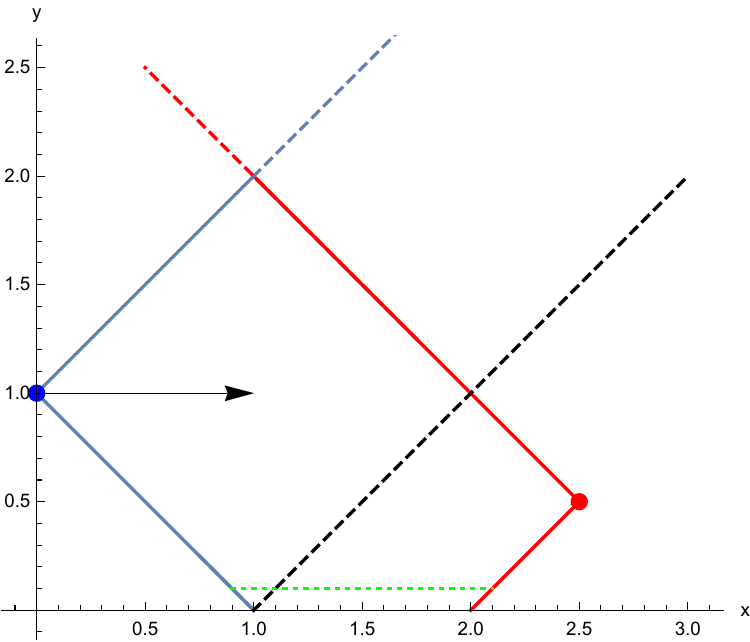}{Case (2.2):  $q_1 \in \A \setminus \cl(M)$ }{fig:att_set25}{5}
{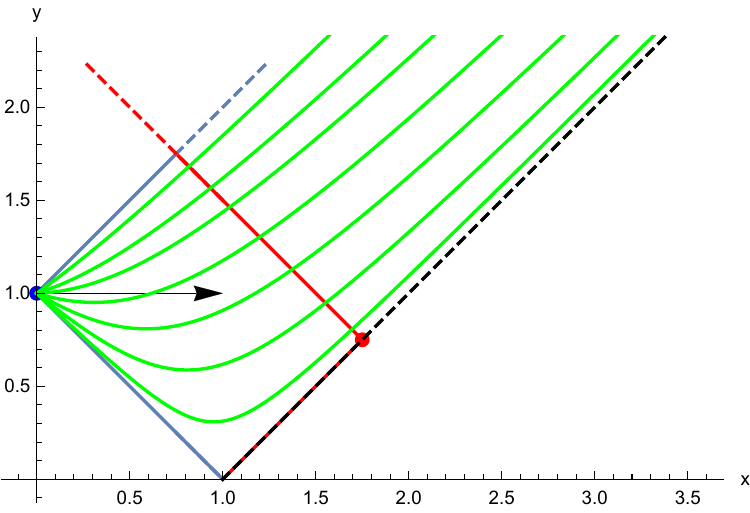}{Case (2.3):  $q_1 \in \A \cap \partial M$  }{fig:att_set26}{5}
}

\begin{lemma}\label{lem:glob_hyp}
The restriction of a negative curvature
 \Lo structure $(g, X_0)$  on $\Aff$ to $D$ is  globally hyperbolic.
\end{lemma}
\begin{proof}
We need to check only condition \eq{compact}. 

Let $q_0, q_1 \in D$. By virtue of Th. \ref{th:J}, the intersection $J^+(q_0) \cap J^-(q_1)$ is either a parallelogram (if $q_1 \in \intt J^+(q_0)$) or a segment (if $q_1 \in \partial J^+(q_0)$) or the empty set (if $q_1 \notin J^+(q_0)$), thus it is compact.
\end{proof}

\begin{theorem}\label{th:exist1}
Let $K<0$, and let $q_0 = \Id$, $q_1 \in J^+$. 
\begin{itemize}
\item[$(1)$]
If $q_1 \in D$, then there exists 
a \Lo length maximizer from $q_0$ to $q_1$. 
\item[$(2)$]
If $q_1 \in E$, then there exist arbitrarily long trajectories from $q_0$ to $q_1$. Thus 
$d(q_1) = + \infty$ and
there are no \Lo length maximizers from $q_0$ to $q_1$.
\end{itemize}
\end{theorem}
\begin{proof}
Item (1) follows from Th. \ref{th:Avez} and Lemma \ref{lem:glob_hyp}.

Item (2).
Take any point $q_1 = (x_1, y_1) \in E$. Denote by $C \in \R^2_{x, y} \setminus G$, $C \neq B$, the intersection point of the lines $\{y = 0\}$ and $\{q \in \R^2 \mid \lam_1(q) = \lam_1(q_1)\}$, see Fig. \ref{fig:att_set25} for the problem $P_1$. Notice that $x(C) > x(B)$. Take any $\eps \in (0, 1)$. Denote by $B_{\eps} \in G$ the intersection point of the lines $\{y = \eps\}$ and $\{q \in \R^2 \mid \lam_2(q) = 0\}$,  and   by $C_{\eps} \in G$ the intersection point of the lines $\{y = \eps\}$ and $\{q \in \R^2 \mid \lam_1(q) = \lam_1(q_1)\}$. The broken line $q_{\eps} = q_0 B_{\eps} C_{\eps} q_1$ is an admissible trajectory of system \eq{pr21}, \eq{pr22} with the cost given by the segment $B_{\eps} C_{\eps}$ only:
$\ds
J(q_{\eps}) = \int_{t(B_{\eps})}^{t(C_{\eps})} \sqrt{|g(u_1, 0)|} dt.
$
For $u_1 = 1$ we get $g(1, 0) = c^2 - a^2 < 0$, $x(t) = x_0 + \eps t$, 
$$
t(C_{\eps}) - t(B_{\eps}) = \frac{x(C_{\eps}) - x(B_{\eps})}{\eps} = \frac{x(C) - x(B) + o(1)}{\eps} \to + \infty 
\text{ \quad as \quad} \eps \to + 0,
$$
thus
$\ds
J(q_{\eps}) = \sqrt{c^2 - a^2}\ \frac{x(C) - x(B) + o(1)}{\eps} \to + \infty 
\text{ as } \eps \to + 0.
$
So $d(q_1) = + \infty$. 
\end{proof}

\begin{remark}
We prove below in Th. $\ref{th:F_nonexist}$ that for any point $q_1 \in F$ there is no \Lo length maximizer connecting $\Id$ to $q_1$.
\end{remark}

\section{Geodesics}\label{sec:extr}
\subsection{Pontryagin maximum principle}
We apply Pontryagin maximum principle (PMP) \cite{PBGM, notes, intro} to optimal control problem \eq{pr21}--\eq{pr24}. 

The Hamiltonian of PMP reads
\begin{align*}
&h_v^{\nu}(\lam) = v_1 h_1(\lam) + v_2 h_2(\lam) - \nu \sqrt{v_1^2-v_2^2}, \qquad
\lam \in T^*G, \qquad \nu \in \R, \\
&h_i(\lam) = \langle \lam, Y_i\rangle, \quad i = 1, 2.
\end{align*}
Since $[Y_1, Y_2] = - \d Y_1 + \g Y_2$, then the Hamiltonian system with the Hamiltonian $h_v^{\nu}$ reads
\begin{align}
&\dot h_1 = -v_2(-\d h_1 + \g h_2), \label{dh1}\\
&\dot h_2 = v_1(-\d h_1 + \g h_2), \label{dh2}\\
&\dot q = v_1 Y_1 + v_2 Y_2. \nonumber
\end{align}

\subsubsection{Abnormal case}
Obvious computations in the abnormal case $\nu = 0$ give the following.

\begin{proposition}\label{propos:abn}
Abnormal extremal trajectories are Lipschitzian reparametrizations of lightlike trajectories:
\begin{align*}
&v_1 = \pm v_2 = 1, \qquad u_1 = \a \pm \b, \ u_2 = \g \pm \d,\\
&q(t) = \exp(t(Y_1\pm Y_2)) = \exp(t(u_1X_1+u_2X_2)),
\end{align*}
these are one-parameter subgroups \eq{1par1}, \eq{1par2}.
\end{proposition}

\subsubsection{Normal case}
Now consider the normal case $\nu = - 1$.
The maximality condition of PMP
\be{max1}
h = v_1h_1+v_2h_2 + \sqrt{v_1^2-v_2^2} \to \max_{v_1 \geq |v_2|}
\ee
yields $h_1^2 - h_2^2 = v_1^2-v_2^2 \equiv 1$. Introduce the hyperbolic coordinates
\begin{align*}
&v_1 = \cosh \f, \quad v_2 = \sinh \f, \qquad \f \in \R, \\
&h_1 = -\cosh \p, \quad h_2 = \sinh \p, \qquad \p \in \R.
\end{align*}
Then the maximality condition \eq{max1} reads
$
h = - \cosh(\f - \p) + 1 \to \max,
$
whence $\f = \p$. Thus the maximized Hamiltonian of PMP reads $H = \frac{-h_1^2+h_2^2}{2}$. Then the vertical subsystem \eq{dh1}, \eq{dh2} of the Hamiltonian system of PMP reduces to the ODE $\dot \p = \d \sinh \p + \g \cosh \p$. Summing up, we have the following description of arclength-parametrized ($g = -v_1^2+v_2^2 \equiv 1$) normal extremals.

\begin{proposition}\label{propos:norm_extr}
Arclength-parametrized normal extremals satisfy the normal Hamiltonian system
\begin{align*}
&\dot \lam = \vH(\lam), \qquad \lam \in T^*G, \\
&H(\lam) = \frac{-h_1^2(\lam)+h_2^2(\lam)}{2} \equiv \frac 12, \qquad h_1(\lam) < 0,
\end{align*}
in coordinates:
\begin{align}
&\dot \p = \d \cosh \p + \g \sinh \p, \label{dpsi1}\\
&\dot q = \cosh \p \, Y_1 + \sinh \p \, Y_2. \label{dq1}
\end{align}
\end{proposition}

Normal extremals are parametrized by covectors
$\ds
\lam_0 \in C = T_{\Id}^*G \cap \{H(\lam) = 1/2, \ h_1(\lam) < 0\}.
$
They are given by the Lorentzian exponential mapping
\be{Exp}
\map{\Exp}{C \times \R_+}{G}, \qquad (\lam_0, t) \mapsto q(t) = \pi \circ e^{t \vH}(\lam_0),
\ee
where $\vH$ is the Hamiltonian vector field on $T^*G$ with the Hamiltonian $H$, $\map{e^{t \vH}}{G}{G}$ is the flow of this vector field, and $\map{\pi}{T^*G}{G}$, $T_q^*G \ni \lam \mapsto q \in G$, is the canonical projection of the cotangent bundle.

\subsection{Parameterization of geodesics}
We integrate ODEs \eq{dpsi1}, \eq{dq1} in the case $\d  \geq 0$, see \eq{a>0}. 
First we integrate the vertical subsystem \eq{dpsi1}:
\be{dpsi2}
\dot \p = \d \cosh \p + \g \sinh \p,  \qquad \p(0) = \p_0, \qquad \d \geq 0.
\ee

\begin{proposition}\label{propos:psi}
Cauchy problem \eq{dpsi2} has the following solutions.
\begin{itemize}
\item[$(1)$]
If $K<0$, then
\begin{align}
&\p(t) = \mu(t) - \theta, \label{mu1}\\
&\d = \D \cosh \theta, \quad \g = \Delta \sinh \theta, \qquad \D = \sqrt{\d^2 - \g^2}, \label{dega1}\\
&\mu(t) = \arsinh \tan \tau, \qquad \tau = \s + \rho, \label{mu(t)1}\\
&\rho = \arctan \sinh(\p_0 + \theta) \in \left(-\frac{\pi}{2}, \frac{\pi}{2}\right), \label{c1}\\
&\s = \D t \in \left(-\frac{\pi}{2} - \rho, \frac{\pi}{2} - \rho\right). \label{si1} 
\end{align}
\item[$(2)$]
If $K>0$, then
\begin{align}
&\p(t) = \mu(t) - \theta, \nonumber\\
&\g = s_1 \D \cosh \theta, \quad \d = s_1 \Delta \sinh \theta, \qquad \D = \sqrt{\g^2 - \d^2}, 
\qquad s_1 = \sgn \g \label{pm2}.
\end{align}
\begin{itemize}
\item[$(2.1)$]
If $\p_0 + \theta = 0$, then $\mu(t) \equiv 0$.
\item[$(2.2)$]
If $\p_0 + \theta \neq 0$, then
\begin{align}
&\mu(t) = \arcosh \coth \tau, \qquad \tau = \rho - \s > 0, \label{mu(t)2}\\
&\s = s_1 \D t < \rho, \label{si2}\\
&\rho = \artanh \cosh(\p_0 + \theta). \label{c2}
\end{align}
\end{itemize}
\item[$(3)$]
If $K = 0$, then
\begin{align*}
&\p(t) = s_1 \mu(t), \\
&\mu(t) = - \ln \tau, \qquad \tau = \rho - \g t >0, \\
&
s_1 = \sgn \g, \qquad  \rho = e^{-s_1 \p_0}.
\end{align*}
\end{itemize}
\end{proposition}
\begin{proof}
(1) Let $K<0$, $\d > 0$. Introduce variables $\D$, $\theta$ according to \eq{dega1}, $\mu$ according to \eq{mu1}, and $\s$ according to \eq{si1}. Then Cauchy problem \eq{dpsi2} transforms to
$$
\der{\mu}{\s} = \cosh \mu, \qquad \mu(0) = \mu_0 = \p_0 + \theta,
$$
which has solution \eq{mu(t)1} by separation of variables. 

Cases (2), (3) are considered similarly.
\end{proof}

Now we integrate the horizontal ODE \eq{dq1} of the Hamiltonian system for normal extremals:
\begin{align}
&\dot x = y k(\p), \qquad k(\p) = \a \cosh \p + \b \sinh \p, \qquad x(0) = 0, \label{dx1}\\
&\dot y = y l(\p), \qquad l(\p) = \g \cosh \p + \d \sinh \p, \qquad y(0) = 1. \label{dy1}
\end{align}

\begin{proposition}\label{propos:hor_param}
Cauchy problem \eq{dx1}, \eq{dy1} has the following solution.
\begin{itemize}
\item[$(1)$]
If $K<0$, then
\begin{align}
&x(t) = \cos \rho \left( \lam (\tan \tau - \tan \rho) + \nu \left(\frac{1}{\cos \tau} - \frac{1}{\cos \rho} \right) \right), \label{x1}\\
&y(t) = \frac{\cos \rho}{\cos \tau}, \label{y1}\\
&\lam = \frac{\a \d - \b \g}{\Delta^2}, \qquad \nu = \frac{\b \d - \a \g}{\Delta^2}, \label{lam1}
\end{align}
where $\rho$, $\tau$, $\Delta$ are defined by \eq{dega1}--\eq{c1}.
The curve $(x(t), y(t))$ is an arc of a hyperbola $y^2 - (w + \sin \rho)^2 = \cos^2 \rho$, where $w = \frac{x-\nu(y-1)}{\lam}$.
\item[$(2)$]
Let  $K>0$.
\begin{itemize}
\item[$(2.1)$]
If $\p_0 + \theta = 0$, then 
\begin{align*}
&x(t) = - \nu (e^{\s}-1),  \\
&y(t)= e^{\s}, \\
&\s = s_1 \D t, \qquad s_1 = \sgn \g, \qquad \D = \sqrt{\g^2-\d^2}. 
\end{align*}
The curve $(x(t), y(t))$ is a line $x + \nu(y-1) = 0$.
\item[$(2.2)$]
If $\p_0 + \theta \neq 0$, then
\begin{align*}
&x(t) = \sinh \rho \left( \nu \left( \frac{1}{\sinh \rho} - \frac{1}{\sinh \tau}\right) + s_2 \lam (\coth \tau - \coth \rho)\right),  \\
&y(t)= \frac{\sinh \rho}{\sinh \tau},  \\
&s_2 = \sgn \mu_0,
\end{align*}
where $\rho$, $\tau$, $\D$ are defined by \eq{pm2}--\eq{c2}.
The curve $(x(t), y(t))$ is an arc of a hyperbola $(s_2 w + \cosh \rho)^2 -y^2 = \sinh^2 \rho$, where $w = \frac{x+\nu(y-1)}{\lam}$.
\end{itemize}
\item[$(3)$]
If $K = 0$, then
\begin{align*}
&x(t) = \rho \left( f (\tau - \rho) + g\left( \frac{1}{\rho} - \frac{1}{\tau} \right)\right), \\
&y(t) = \frac{\rho}{\tau}, \\
&f = - \frac{\a - s_1 \b}{2 \g}, \qquad g = - \frac{\a + s_1 \b}{2 \g}, \qquad s_1 = \sgn \g, \\
&\tau = \rho - \g t, \qquad \rho = e^{-s_1 \p_0}.
\end{align*}
The curve $(x(t), y(t))$ is an arc of a hyperbola $w =  \rho^2 \left( \frac 1y - 1\right)$, where $w = \frac{x+g(y-1)}{f}$.
\end{itemize}
\end{proposition}
\begin{proof}
Cauchy problem \eq{dx1}, \eq{dy1} integrates as
\begin{align}
&x(t) = K(t) = \int_0^t k(s) e^{L(s)} ds, \label{x2}\\
&y(t) = e^{L(t)}, \qquad L(t) = \int_0^t l(s) ds. \label{y2}
\end{align}

(1) Let $K<0$, $\d > 0$. By item (1) of Propos. \ref{propos:psi},
\begin{align}
&\cosh \p = \frac{1}{\D} \left( \frac{\d}{\cos \tau} - \g \tan \tau \right), \qquad
\sinh \p = \frac{1}{\D} \left( \d \tan \tau - \frac{\g}{\cos \tau}\right), \nonumber\\
&k = \D \tan \tau,  \qquad
l = \D \left( \frac{\lam}{\cos \tau} + \nu \tan \tau\right), \label{l1}
\end{align}
and formulas \eq{x1}, \eq{y1} follow from \eq{x2}--\eq{l1}.

(2), (3) The cases $K\geq 0$ are treated similarly.
\end{proof}

\subsection{Geodesic completeness}
Denote the maximal domain of a solution $\lam_t$ to a Cauchy problem $\dot \lam = \vH(\lam)$, $\lam(0) = \lam_0 \in C$ as
$\left(t_{\min}(\lam_0), t_{\max}(\lam_0)\right) \ni 0$. We obtain the following explicit description of this domain from Propositions \ref{propos:psi}  and \ref{propos:hor_param}.

\begin{corollary}\label{cor:tminmax}
\begin{itemize}
\item[$(1)$]
If $K<0$, then $\ds t_{\min} = -\frac{\pi/2 + \rho}{\D}$, $\ds t_{\max} = \frac{\pi/2 - \rho}{\D}$.
\item[$(2)$]
If $K>0$, then:
\begin{itemize}
\item[$(2.1)$]
if $\p_0 + \theta = 0$, then $t_{\min} = -\infty$, $t_{\max} = + \infty$,
\item[$(2.2)$]
if $\p_0 + \theta \neq 0$, then 
$
\begin{cases}
\ds t_{\min} = -\infty, \qquad t_{\max} = \frac{\rho}{\D} \qquad &\text{for } \g > 0,\\
\ds t_{\min} = -\frac{\rho}{\D}, \qquad t_{\max} = + \infty \qquad &\text{for } \g < 0.
\end{cases}
$
\end{itemize}
\item[$(3)$]
If $K = 0$, then
$
\begin{cases}
\ds t_{\min} = -\infty, \qquad t_{\max} = \frac{\rho}{\g} \qquad &\text{for } \g > 0,\\
\ds t_{\min} = \frac{\rho}{\g}, \qquad t_{\max} = + \infty \qquad &\text{for } \g < 0.
\end{cases}
$
\end{itemize}
\end{corollary}

We recall standard definitions of \Lo geometry related to geodesic completeness \cite{beem}.

A timelike arclength-parametrized geodesic $q(t)$ in a \Lo manifold is called complete if it can be extended to be defined for $-\infty < t < + \infty$, otherwise it is called incomplete. Future and past complete (incomplete) geodesics are defined similarly.

A \Lo manifold $M$ is called  
timelike geodesically complete if all timelike arclength-parametrized geodesics are complete, otherwise $M$ is called  
timelike geodesically incomplete. Future and past timelike geodesically complete (incomplete) \Lo manifolds  are defined similarly.

Now Corollary \ref{cor:tminmax} implies the following.

\begin{corollary}\label{cor:geod_compl}
If $K<0$, then $\Aff$ is both future and past timelike geodesically incomplete.

Let  $K\geq 0$. If $\g > 0$, then $\Aff$ is past timelike geodesically complete and future   timelike geodesically incomplete. If $\g < 0$, then $\Aff$ is past timelike geodesically incomplete and future   timelike geodesically complete.

Thus in all cases $\Aff$ is   timelike geodesically incomplete.
\end{corollary}

\section{\Lo length maximizers}\label{sec:max}
We prove that all extremal trajectories described in Sec. \ref{sec:extr} are optimal, i.e., are \Lo length maximizers. The main tool is the following Hadamard's global diffeomorphism theorem.

\begin{theorem}[Th. 6.2.8~\cite{hadamard}]\label{th:Had}
Let $X$, $Y$ be smooth manifolds and let $\map{F}{X}{Y}$ be a smooth mapping such that:
\begin{enumerate}
\item
$\dim X = \dim Y$, 
\item
$X$ and  $Y$ are arcwise connected,
\item
$Y$ is simply connected,
\item
$F$ is nondegenerate,
\item
$F$ is proper (i.e., preimage of a compact is a compact).
\end{enumerate}
Then $F$ is a diffeomorphism.
\end{theorem}

\subsection{Diffeomorphic properties of the exponential mapping}
Denote the following open subset $M \subset G$:
\begin{align*}
&K<0 \then M = \intt D, \\
&K\geq 0 \then M = \intt J^+.
\end{align*}
The set $M \cong \R^2$ will serve as the domain of the exponential mapping $\map{\Exp}{N}{G}$, in view of the following theorem.

\begin{theorem}\label{th:opt1}
\begin{itemize}
\item[$(1)$]
$\Exp(N) \subset M$.
\item[$(2)$]
$\map{\Exp}{N}{M}$ is a diffeomorphism.
\item[$(3)$]
For any $\lam_0 \in N$ and any $t_1 \in (0, t_{\max}(\lam_0))$ the extremal trajectory $\Exp(\lam_0,t)$, $t \in [0,t_1]$, is optimal.  
\end{itemize}
\end{theorem}
\begin{proof}
We consider only the  case $K<0$ since the case $K\geq 0$ is more simple and are treated similarly.
So let $K<0$, then
\begin{align}
&M =  \intt D = \{ q \in G \mid \lam_1(q) < 0 < \lam_2(q), \ \lam_3(q) > 0 \}, \label{M1}\\
& N= \left\{(\rho, \tau) \in \R^2 \mid \rho \in \left(-\frac{\pi}{2},\frac{\pi}{2}\right), \ \tau \in \left(\rho, \frac{\pi}{2}\right)  \right\}.  \label{N1}
\end{align}
Since $\d> 0$ by virtue of \eq{a>0} and $\d^2 - \g^2> 0$ by virtue of $K<0$, then $\d > |\g|$.  Further, we have factorizations along arclength-parametrized timelike geodesics $(x(t), y(t))$ given by item (1) of Propos. \ref{propos:hor_param}:
\begin{align}
&\lam_1(x(t), y(t)) = \frac{2}{\d - \g} \frac{\sin\left(\frac{\pi}{4}-\frac{\rho}{2}\right) 
\sin\left( \frac{\rho-\tau}{2}\right)}
{ \sin\left(\frac{\pi}{4}+\frac{\tau}{2}\right)}, \label{f1}\\
&\lam_2(x(t), y(t)) = -\frac{2}{\d + \g} \frac{\sin\left(\frac{\pi}{4}+\frac{\rho}{2}\right) 
\sin\left( \frac{\rho-\tau}{2}\right)}
{ \sin\left(\frac{\pi}{4}-\frac{\tau}{2}\right)}, \label{f2}\\
&\lam_3(x(t), y(t)) = \frac{2}{\d - \g} \frac{\sin\left(\frac{\pi}{4}+\frac{\rho}{2}\right) 
\cos\left( \frac{\rho-\tau}{2}\right)}
{ \sin\left(\frac{\pi}{4}+\frac{\tau}{2}\right)}. \label{f3}
\end{align}

(1) Factorizations \eq{f1}--\eq{f3} and equalities \eq{M1}, \eq{N1} imply immediately that $\Exp(N) \subset M$.

(2) We apply Th. \ref{th:Had} to the mapping $\map{\Exp}{N}{M}$. Both $N$ and $M$ are diffeomorphic to $\R^2$.
The Jacobian of the exponential mapping is
$\ds
\pder{(x, y)}{(\tau, \rho)} = - \lam \frac{\cos \rho \sin(\rho - \tau)}{\cos^2 \tau} < 0 \text{ on } N,
$
thus $\map{\Exp}{N}{M}$ is nondegenerate. Finally, factorizations \eq{f1}--\eq{f3} imply that if $(\rho, \tau) \to \partial N$, then $(x, y) = \Exp(\rho, \tau) \to \partial M$, thus $\map{\Exp}{N}{M}$ is proper. Consequently, $\map{\Exp}{N}{M}$ is a diffeomorphism.

(3) Let $\lam_0 \in N$, and let $t_1 \in (0, t_{\max}(\lam_0))$. Let us prove that the trajectory $q(t) = \Exp(\lam_0, t)$, $t \in [0, t_1]$, is optimal. 
We have $q_1 = q(t_1) = \Exp(\lam, t_1) \in M$. Moreover, by item (2) of this theorem $q(t)$, $t \in [0, t_1]$, is a unique arclength-parametrized geodesic connecting $\Id$ to $q_1$. By item (1) of Th. \ref{th:exist1} there exists an optimal trajectory connecting these points, so it coincides with  $q(t)$, $t \in [0, t_1]$.

\end{proof}

\subsection{Inverse of the exponential mapping and optimal synthesis}

\begin{theorem}\label{th:inverse}
The inverse of the exponential mapping $\map{\Exp^{-1}}{M}{N}$, $(x_1, y_1) \mapsto (\p_0, t_1)$ is given as follows.
\begin{itemize}
\item[$(1)$]
If $K<0$, then
\begin{align}
&t_1 = \frac{\tau- \rho}{\D}, \qquad \p_0 = \arsinh \tan \rho - \theta, \label{t1}\\
&\tau = \arcsin\left(\frac{y_1^2 + w^2 - 1}{2 y_1 w} \right), \qquad 
\rho = \arcsin\left(\frac{y_1^2 - w^2 - 1}{2 w} \right), \qquad
w = \frac{x_1 - \nu(y_1-1)}{\lam}. \label{tau1D}
\end{align} 
\item[$(2)$]
 Let $K>0$, and let 
\be{ws1}
\ds w = \frac{x_1 + \nu(y_1-1)}{\lam}, \qquad s_1 = \sgn \g.
\ee
\begin{itemize}
\item[$(2.1)$]
If $w = 0$, then $\ds t_1 = s_1 \frac{\ln y_1}{\D}$, $\p_0 = - \theta$.
\item[$(2.2)$]
If $w \neq  0$, then
\begin{align}
&t_1 = s_1 \frac{\rho- \tau}{\D}, \qquad \p_0 = \arcosh \coth \rho - \theta, \nonumber\\
&\tau = \arcosh\left(s_2 \frac{1-y_1^2 - w^2}{2 y_1 w} \right), \qquad 
\rho = \arcosh\left(s_2 \frac{1-y_1^2 + w^2}{2 w} \right), \qquad
s_2 = \sgn(\lam w). \label{taurho2}
\end{align}
\end{itemize}

\item[$(3)$]
If $K = 0$, then
\begin{align}
&t_1 = \frac{\rho - \tau}{\g}, \qquad \p_0 = - s_1 \ln \rho, \nonumber\\
&\tau = \sqrt{\frac{w}{y_1 - y_1^2}}, \qquad \rho = \sqrt{\frac{w y_1}{1 - y_1}}, \label{taurho3} \\
&w = \frac{x_1 - g(1-y_1)}{f}, \qquad
f = - \frac{\a - s_1 \b}{2 \g}, \qquad g = - \frac{\a + s_1 \b}{2 \g}, \qquad s_1 = \sgn \g. \nonumber
\end{align}   
\end{itemize}

For any $(x_1, y_1) \in M$,
there is a unique  arclength-parametrized optimal trajectory connecting $\Id$ to $(x_1, y_1)$, and it is $q(t) = \Exp(\p_0, t)$, $t\in [0, t_1]$.
\end{theorem}
\begin{proof}
We consider only the case $K<0$. Then the parametrization of \Lo geodesics given by item (1) of Propos. \ref{propos:hor_param} yields
\begin{align*}
&\sin \rho =  {y_1 \sin \tau-w} , \qquad \cos \rho =  y_1 \cos \tau, \\
&1 = \sin^2 \rho + \cos^2 \rho = y_1^2 - 2 y_1 w \sin \tau, \\
&\sin \tau  = \frac{y_1^2 + w^2-1}{2 y_1 w}, \qquad \sin \rho = \frac{y_1^2 - w^2-1}{2 w},
\end{align*}
and formulas of item (1) of this theorem follow since $\tau, \rho \in \left(-\frac{\pi}{2},\frac{\pi}{2}\right)$. 
\end{proof}

\begin{theorem}\label{th:F_nonexist}
Let $K<0$. If $q_1 \in F$, then there is no \Lo length maximizer connecting $q_0$ to $q_1$.
\end{theorem}
\begin{proof}
Lightlike extremal trajectories starting at $q_0$ fill the set $\partial J^+ = \{q \in G \mid \lam_1(q) \lam_2(q) = 0\}$. 
By item (1) of Th.~\ref{th:opt1}, timelike extremal trajectories starting at $q_0$ fill the domain $\intt D = \{ q \in G \mid \lam_1(q) < 0 < \lam_2(q), \ \lam_3(q) > 0\}$.  Thus extremal trajectories starting at $q_0$ do not intersect the set $F = \{q \in G \mid  \lam_3(q) = 0\}$. By PMP, there is no optimal trajectory connecting $q_0$ to a point $q_1 \in F$.
\end{proof}

\begin{remark}
The reasoning of the preceding theorem applied to the set $E = \{q \in G \mid   \lam_3(q) > 0\}$ proves once more that there are no \Lo length maximizers connecting $q_0$ to points in $E$, in addition to item $(2)$ of Th. $\ref{th:exist1}$.
\end{remark}

\section{\Lo distance and spheres}\label{sec:dist}
We describe explicitly the \Lo distance $d(q) = d(\Id, q)$ and spheres
$
S(R) = \{q \in G \mid d(q) = R\}$, $R \in [0, + \infty]$.

\subsection{The case $K<0$}

\begin{theorem}\label{th:d1}
Let $K<0$ and let $q_1 = (x_1, y_1) \in G$.
\begin{itemize}
\item[$(1)$]
If $q_1 \notin J^+$, then $d(q_1) = 0$.
\item[$(2)$]
If $q_1 \in \partial J^+$, then $d(q_1) = 0$.
\item[$(3)$]
If $q_1 \in \intt D$, then $\ds d(q_1) = \frac{\tau - \rho}{\D}$, where  $\tau$, $\rho$ are given by \eq{tau1D}. In particular,
\be{dintD}
{d}({\intt D}) = \left(0, \frac{\pi}{\D} \right).
\ee
\item[$(4)$]
If $q_1 \in F$, then $\ds d(q_1) = \frac{\pi}{\D}$.
\item[$(5)$]
If $q_1 \in E$, then $d(q_1) = + \infty$.
\end{itemize}
\end{theorem}
\begin{proof}
(1) follows from the definition of \Lo distance $d$.

(2) follows since the only trajectories connecting $\Id$ to $q_1 \in \partial J^+$ are lightlike by item (1) of Th. \ref{th:opt1}.

(3) follows from item (1) of Th. \ref{th:opt1}.

(4) Let $q_1 \in F$. Take any sequence $(\tau^n, \rho^n) \in N$ such that $\tau^n \to \frac{\pi}{2} - 0$, $\rho^n \to - \frac{\pi}{2} + 0$, $\frac{\tau^n + \pi/2}{\rho^n + \pi/2} \to + \infty$. Then the parametrization of the exponential mapping \eq{x1}, \eq{y1} implies that the point $q^n = \Exp(\tau^n, \rho^n) \in \intt D$ and $q^n \to B = \{y = \lam_3(q) = 0\}$. By item (3) of this theorem, $d(q^n) = \frac{\tau^n - \rho^n}{\D} \to \frac{\pi}{\D}$.

Considering a trajectory of the field $X_1 = y \pder{}{x}$ starting at $q^n$ and terminating at the ray $F$, we get the bound $\restr{d}{F} \geq \frac{\pi}{\D}$.

Now we show that in fact $\restr{d}{F} = \frac{\pi}{\D}$. To this end we cite the following basic property of \Lo distance.

\begin{lemma}[Lemma 4.4 \cite{beem}]\label{lem:beem}
For \Lo distance $d$ on a \Lo manifold, if $d(p, q) < \infty$, $p_n \to p$, and $q_n \to q$, then $d(p, q) \leq \liminf_{n\to \infty} d(p_n, q_n)$. 

Also, if $d(p, q) = \infty$, $p_n \to p$, and $q_n \to q$, then $\lim_{n\to \infty} d(p_n, q_n) = \infty$.
\end{lemma}

Take any point $\bq \in F$. Choose any sequence $\intt D \ni q^n \to \bq$. If $d(\bq) = \infty$, then Lemma \ref{lem:beem} implies $\lim_{n\to \infty} d(q_0, q^n) = \infty$, which contradicts the bound \eq{dintD}. Thus $d(\bq) < \infty$. Then by Lemma \ref{lem:beem} $d(q_0, \bq) \leq \liminf_{n\to \infty} d(q_0, q^n) \leq \frac{\pi}{\D}$. So $d(q_0, \bq) = \frac{\pi}{\D}$.

(5) follows from item (2) of Th. \ref{th:exist1}.

\end{proof}

The explicit description of \Lo length maximizers given by Th. \ref{th:inverse} implies, via transformations of elementary functions, the following characterization of \Lo spheres centred at $\Id$.

\begin{corollary}\label{cor:sphere1}
Let $K<0$.
\begin{itemize}
\item[$(1)$]
$S(0) = \{ q \in G \mid \lam_1(q) \geq 0 \text{ or } \lam_2(q) \leq 0\}$.
\item[$(2)$]
If $R \in (0, \frac{\pi}{\D})$, then
\begin{align*}
&S(R) = \{(x, y) \in G \mid w^2 - (y-\cos \s)^2 = \sin^2 \s\}, \qquad
w = \frac{x - \nu(y-1)}{\lam}, \qquad \s = \D R,
\end{align*}
it is an arc of a hyperbola noncompact in both directions.
\item[$(3)$]
$S(\frac{\pi}{\D}) = F$.
\item[$(4)$]
If $R \in (\frac{\pi}{\D}, + \infty)$, then
$S(R) = \emptyset$.
\item[$(5)$]
$S(+\infty) =   E$.
\end{itemize}
\end{corollary}

\subsection{The case $K>0$}
\begin{theorem}\label{th:d2}
Let $K>0$ and let $q_1 = (x_1, y_1) \in G$.
\begin{itemize}
\item[$(1)$]
If $q_1 \notin J^+$, then $d(q_1) = 0$.
\item[$(2)$]
If $q_1 \in \partial J^+$, then $d(q_1) = 0$.
\item[$(3)$]
If $q_1 \in \intt J^+ \cap \{w \neq 0\}$, 
then $\ds d(q_1) = s_1 \frac{\rho- \tau  }{\D}$, where $s_1$, $w$, $\tau$, $\rho$ are given by \eq{ws1}, \eq{taurho2}. In particular,
$
{d}({\intt J^+} \cap \{w \neq 0\}) = \left(0, + \infty \right).
$
\item[$(4)$]
If $q_1 \in \intt J^+ \cap \{w = 0\}$, 
then $\ds d(q_1) = s_1 \frac{\ln y_1  }{\D}$, where $s_1$ is  given by \eq{ws1}. In particular,
$
{d}({\intt J^+} \cap \{w = 0\}) = \left(0, + \infty \right).
$
\end{itemize}
\end{theorem}
\begin{proof}
Similarly to the proof of Th. \ref{th:d1}.
\end{proof}

\begin{corollary}\label{cor:sphere2}
Let $K>0$.
\begin{itemize}
\item[$(1)$]
$S(0) = \{ q \in G \mid \lam_1(q) \geq 0 \text{ or } \lam_2(q) \leq 0\}$.
\item[$(2)$]
If $R \in (0, + \infty)$, then
\begin{align*}
&S(R) = \{(x, y) \in G \mid (y-\cosh \s)^2  - w^2  = \sinh^2 \s\}, \qquad
w = \frac{x + \nu(y-1)}{\lam}, \qquad \s = s_1 \D R, \qquad s_1 = \sgn \d,
\end{align*}
it is an arc of a hyperbola noncompact in both directions.
\item[$(3)$]
$S(+\infty) =   \emptyset$.
\end{itemize}
\end{corollary}
\begin{proof}
Similarly to the proof of Cor. \ref{cor:sphere1}.
\end{proof}

\subsection{The case $K=0$}
\begin{theorem}\label{th:d3}
Let $K = 0$ and let $q_1\in G$.
\begin{itemize}
\item[$(1)$]
If $q_1 \notin J^+$, then $d(q_1) = 0$.
\item[$(2)$]
If $q_1 \in \partial J^+$, then $d(q_1) = 0$.
\item[$(3)$]
If $q_1 \in \intt J^+$, 
then $\ds d(q_1) = \frac{\rho- \tau  }{\g}$, where $\tau$, $\rho$ are given by \eq{taurho3}. In particular,
$
{d}({\intt J^+}) = \left(0, + \infty \right).
$
\end{itemize}
\end{theorem}
\begin{proof}
Similarly to the proof of Th. \ref{th:d1}.
\end{proof}

\begin{corollary}\label{cor:sphere3}
Let $K = 0$.
\begin{itemize}
\item[$(1)$]
$S(0) = \{ q \in G \mid \lam_1(q) \geq 0 \text{ or } \lam_2(q) \leq 0\}$.
\item[$(2)$]
If $R \in (0, +\infty)$, then
\begin{align*}
&S(R) = \{(x, y) \in G \mid (w + \s^2)y = w\}, \qquad
w = \frac{x + g(y-1)}{f}, \qquad \s = \g R,
\end{align*}
it is an arc of a hyperbola noncompact in both directions.
\item[$(3)$]
$S(+\infty) =   \emptyset$.
\end{itemize}
\end{corollary}
\begin{proof}
Similarly to the proof of Cor. \ref{cor:sphere1}.
\end{proof}

\subsection{Regularity of \Lo distance}
\begin{corollary}\label{cor:regul}
We have
 $d \in C^{\om}(M) \cap C(\cl D)$.
\end{corollary}
\begin{proof}
We consider only the case $K<0$.
If $q_1 \in M = \intt D$, then item (3) of Th. \ref{th:d1} gives $d(q_1) = \frac{\tau_1 - \rho_1}{\D}$, and the functions $\tau_1$, $\rho_1$ are real-analytic since $\map{\Exp^{-1}}{\intt D}{N}$ is real-analytic by virtue of the inverse function theorem for real-analytic mappings.

In order to show the inclusion $d \in C(\cl D)$, it remains to prove continuity of $d$ on the boundary $\partial D = \partial J^+ \cup F$. If $\intt D \ni q^n \to q_1 \in \partial J^+$, then by virtue of items (2), (3) of Th. \ref{th:d1} we have $d(q^n) \to 0 = d(q_1)$. And if $\intt D \ni q^n \to q_1 \in F$, then similarly $d(q^n) \to \frac{\pi}{\D} = d(q_1)$.
\end{proof}

Now we study asymptotics of the \Lo distance $d(q)$ near the boundary of the domain $M$.
For a point $q \in M$, denote by $d_M(q)$ the Euclidean distance from $q$ to $\partial M$.
The explicit expression for the \Lo distance in the domain $M$ given by Theorems \ref{th:d1}--\ref{th:d3} implies that near smoothness points of $\partial M$ the distance $d(q)$ is H\"older with exponent $\frac 12$ of the distance $d_M(q)$, similarly to the Minkowski plane.

\begin{corollary}\label{cor:as}
Let $\bq \in \partial M$ be a point of smoothness of the curve $\partial M$. Then
\begin{align*}
&d(q) = d(\bq) + f(\bq) \sqrt{d_M(q)} + O(d_M(q))^{3/2},\\
&M \ni q \to \bq, \qquad f(\bq) \neq 0.
\end{align*}
\end{corollary}

\begin{remark}
Alternative proofs of Corollaries {\em \ref{cor:regul}}, {\em \ref{cor:as}} follow by local isometry of $\Aff$ with standard constant curvature \Lo manifolds $\R^2_1$, $\SS^2_1$, $\widetilde{\H^2_1}$.
\end{remark}

\section{Isometries}\label{sec:isom}

\subsection{Infinitesimal isometries of \Lo manifolds}

We recall some necessary facts of \Lo (in fact, pseudo-Riemannian geometry) \cite{Neill}.

A vector field $X$ on a \Lo manifold $(M, g)$ is called a Killing vector field (or an infinitesimal isometry) if $L_X g = 0$. 

\begin{proposition}[\cite{Neill}, Propos. 23]
A vector field $X$  is Killing iff
the mappings $\p_t$ of its local flow satisfy $\p_t^* g = g$, where $\map{\p_t}{M}{M}$ is the shift of $M$ along $X$ by time $t$.
\end{proposition}

\begin{corollary}
A vector field $X$  is Killing iff $d(q_1, q_2) = d(\p_t(q_1), \p_t(q_2))$ for all $q_1, q_2 \in M$ and all $t$ for which the right-hand side is defined.
\end{corollary}

\begin{proposition}[\cite{Neill}, Propos. 25]\label{propos:Xg}
A vector field $X$  is Killing iff
\be{Xg}
Xg(V, W) = g([X, V], W) + g(V, [X, W]), \qquad V, W \in \VEC(M).
\ee
\end{proposition}

Denote by $i(M)$ the set of Killing vector fields on a \Lo manifold $M$. The set $i(M)$ is a Lie algebra over $\R$ w.r.t. Lie bracket of vector fields.

\begin{lemma}[\cite{Neill}, Lemma 28]
The Lie algebra $i(M)$ on a connected \Lo manifold $M$, $\dim M = n$, has dimension at most $\frac{n(n+1)}{2}$.
\end{lemma}

\begin{remark}\label{rem:dimi}
Let $M$ be a connected \Lo manifold of dimension $n$. Then $\dim i(M) = \frac{n(n+1)}{2}$ iff $M$ has constant curvature (Exercises $14$, $15$ {\em\cite{Neill}}). 
\end{remark}

Denote by $I(M)$ the set of all isometries of a \Lo manifold $M$.

\begin{theorem}[\cite{Neill}, Theorem 32]
$I(M)$ is a Lie group.
\end{theorem}

Denote by $ci(M)$ the set of all complete Killing vector fields on $M$.

\begin{proposition}[\cite{Neill}, Propos. 33]\label{propos:I}
\begin{itemize}
\item[$(1)$]
$ci(M)$ is a Lie subalgebra of $i(M)$.
\item[$(2)$]
There is a Lie anti-isomorphism between the Lie algebra of the Lie group $I(M)$ and the Lie algebra $ci(M)$.
\end{itemize}
\end{proposition}

Denote by $I_0(M)$ the connected component of the identity in the Lie group $I(M)$.

\subsection{Killing vector fields and isometries of $\Aff$}
We compute the Lie algebra of Killing vector fields for left-invariant \Lo structures on $G = \Aff$.

By Th. \ref{th:K}, such \Lo structures have constant curvature. By Remark \ref{rem:dimi}, 
\be{dimiG}
\dim i(G) = 3.
\ee
Left translations on the Lie group $G$  are obvious isometries. They are generated by right-invariant vector fields on~$G$:
$$
\tX_1(q) = R_{q*}X_1(\Id) = \pder{}{x}, \qquad
\tX_2(q) = R_{q*}X_2(\Id) = x\pder{}{x}+y\pder{}{y},
$$
where $R_q \ : \ \bq \mapsto \bq q $ is the right translation on $G$. Since $[\tX_i, X_j] = 0$, Propos. \ref{Xg} implies that $\tX_1$, $\tX_2$ are Killing vector fields. By virtue of \eq{dimiG}, in order to describe the 3D Lie algebra $i(G)$ it remains to find just one Killing vector field linearly independent on $\tX_1$, $\tX_2$. 

\begin{lemma}
If $X \in \VEC(G)$ is a Killing vector field such that $X(\Id) = 0$, then $X$ is tangent to \Lo spheres $S(R)$, $R \in [0, + \infty]$.
\end{lemma}
\begin{proof}
Local flow of $X$ preserves the \Lo distance $d(\Id, q)$, thus the \Lo spheres as well.
\end{proof}

\begin{lemma}\label{lem:X-+0}
The following vector field is tangent to \Lo spheres $S(R)$, $R \in (0, + \infty)$:
\begin{itemize}
\item[$(1)$]
$K<0 \then X_- = \left(y^2+w^2\right)\pder{}{w} + 2 w y \pder{}{y} = 
\left(\lam(y^2+w^2-1) + 2 \nu w y\right)\pder{}{x} + 2 w y \pder{}{y}$, $w = \frac{x - \nu (y-1)}{\lam}$,
\item[$(2)$]
$K>0 \then X_+ = \left(y^2+w^2\right)\pder{}{w} + 2 w y \pder{}{y} = 
\left(\lam(y^2+w^2-1) - 2 \nu w y\right)\pder{}{x} + 2 w y \pder{}{y}$, $w = \frac{x + \nu (y-1)}{\lam}$,
\item[$(3)$]
$K=0 \then X_0 = w\pder{}{w} +  y(1-y) \pder{}{y} = 
\left(x + g(y^2-1)\right)\pder{}{x} + y(1-y)\pder{}{y}$, $w = \frac{x + g (y-1)}{f}$.
\end{itemize}
\end{lemma}
\begin{proof}
Follows from the explicit parametrization of the spheres $S(R)$, $R \in (0, + \infty)$, see Corollaries \ref{cor:sphere1}, \ref{cor:sphere2}, \ref{cor:sphere3} respectively. 
\end{proof}

\begin{theorem}\label{th:iKneq0}
Let $K \neq 0$. Then $i(G) = \spann(\tX_1, \tX_2, X_{\pm})$, where $\pm = \sgn K$, and
$X_{\pm}$ is given by items $(1)$, $(2)$ of Lemma $\ref{lem:X-+0}$. The table of Lie brackets in this Lie algebra is $[\tX_1, \tX_2] = \tX_1$, $[\tX_1, X_{\pm}] = \mp \frac{2 \nu}{\lam}\tX_1 + \frac{2}{\lam} \tX_2$, $[\tX_2, X_{\pm}] = \frac{2(\lam^2-\nu^2)}{\lam} \tX_1 \pm \frac{2\nu}{\lam} \tX_2 + X_{\pm}$. The Lie algebra $i(G)$ is isomorphic to the Lie algebra $\sl(2)$ of the Lie group $\SL(2)$ of unimodular $2 \times 2$ matrices.  
\end{theorem}
\begin{proof}
The vector field $X_{\pm}$ satisfies identity \eq{Xg}, thus it is Killing. Since $X_{\pm}$ is linearly independent of $\tX_1$, $\tX_2$ and $\dim i(G) = 3$, it follows that $i(G) = \spann(\tX_1, \tX_2, X_{\pm})$. 
The table of Lie brackets in this Lie algebra is verified immediately. Moreover, these Lie brackets imply that the Lie algebra $i(G)$ is simple, thus it is isomorphic to $\sl(2)$ or $\so(3)$, see the classification of 3D Lie algebras in \cite{jacobson}. But $i(G)$ contains a 2D Lie subalgebra spanned by $\tX_1$, $\tX_2$, which is impossible in $\so(3)$. Thus $i(G) \cong \sl(2)$.
\end{proof}

\begin{theorem}
Let $K = 0$. Then $i(G) = \spann(\tX_1, \tX_2, X_0)$, where $X_0$ is given by item $(3)$ of Lemma $\ref{lem:X-+0}$. The table of Lie brackets in this Lie algebra is $[\tX_1, \tX_2] = \tX_1$, $[\tX_1, X_0] = \tX_1$, $[\tX_2, X_0] = 2 g \tX_1 - \tX_2 + X_0$. The Lie algebra $i(G)$ is isomorphic to the Lie algebra $\sh(2)$ of the Lie group $\SH(2)$ of hyperbolic motions of the plane.  
\end{theorem}
\begin{proof}
Similarly to the proof of Th. \ref{th:iKneq0}.
\end{proof}

\begin{proposition}
\begin{itemize}
\item[$(1)$]
$ci(G) = \spann(\tX_1, \tX_2)$.
\item[$(2)$]
$I_0(\Aff) = \{ L_q \mid q \in \Aff\} \cong \Aff$.
\end{itemize}
\end{proposition}
\begin{proof}
Item (1).
The vector fields $\tX_1$, $\tX_2$ are complete. Although, each vector field $X_0$, $X_{\pm}$ is not complete.

Item (2). By virtue of Propos.
\ref{propos:I} and item (1) of this proposition, the Lie algebra of the Lie group $I_0(\Aff)$ is anti-isomorphic to $ci(G) = \spann(\tX_1, \tX_2)$. 
\end{proof}

\subsection{Isometric embedding of $\Aff$ into $\R^2_1$ in the case $K = 0$}

\begin{theorem}\label{th:iso}
Let $K = 0$.
The mapping $\map{i}{\Aff}{\Pi \subset \R^2_1}$, $\Pi = \left\{(\tx, \ty) \in \R^2_1 \mid s_1 \ty + \tx < 1/\gamma \right\}$,
\be{**}
i(x, y) = (\tx, \ty) = \left( \frac 12 \left(\frac{y-1}{y} - \frac{w}{\gamma} \right), \frac{s_1}{2} \left(\frac{y-1}{y} + \frac{w}{\gamma}\right)\right),
\ee
is an isometry.
\end{theorem}
\begin{proof}
We give a proof for the problem $P_3$, in the general case $K= 0$ the proof is similar.

For the problem $P_3$ we have $\Pi = \left\{(\tx, \ty) \in \R^2_1 \mid \ty + \tx < 1 \right\}$, 
\be{i}
i(x, y) = (\tx, \ty) = \left( \frac 12 \left(1 - \frac 1y + x\right), \frac{s_1}{2} \left(1 - \frac 1y + x\right)\right).
\ee
Let $q_j= (x_j, y_j) \in \Aff$, $\tq_j = i(q_j) = (\tx_j, \ty_j) \in \R^2_1$, $j = 1, 2$. Immediate computation on the basis of \eq{i} shows that $\tq_j \in \Pi$, $j = 1, 2$. We prove that
\be{tidd}
\td(\tq_1, \tq_2) = d(q_1, q_2),
\ee
where $d$ and $\td$ are the \Lo distances in $\Aff$ and $\R^2_1$ respectively.

First we show that
\be{dtid}
d(q_1, q_2) \neq 0 \iff \td(\tq_1, \tq_2) \neq 0.
\ee
Denote $\bq = q_1^{-1} q = (\bx, \by) = ((x_2-x_1)/y_1, y_2/y_1)$. Then
$$
d(q_1, q_2) \neq 0   \iff d(\Id, \bq) \neq 0 \iff \bx > 0, \ \by > 1 \iff x_2>x_1, \ y_2 > y_1.
$$
On the other hand,
\begin{align*}
\td(\tq_1, \tq_2) \neq 0 
&\iff \tx_2 - \tx_1 > |\ty_2-\ty_1| \iff
\begin{cases}
\tx_2 - \tx_1 > \ty_2-\ty_1,\\
 \tx_2 - \tx_1 > \ty_1-\ty_2
\end{cases}\\
&\iff
\begin{cases}
x_2 - \frac{1}{y_2} - x_1 + \frac{1}{y_1} > - x_1 + x_2 - \frac{1}{y_1} + \frac{1}{y_2},\\
x_2 - \frac{1}{y_2} - x_1 + \frac{1}{y_1} > x_1 - x_2 + \frac{1}{y_1} - \frac{1}{y_2}
\end{cases}
\iff
\begin{cases}
\frac{1}{y_1} > \frac{1}{y_2},\\
x_1 - x_2 < 0,
\end{cases}
\end{align*}
and \eq{dtid} follows.

Now let $d(q_1, q_2) \neq 0$, $\td(\tq_1, \tq_2) \neq 0$, and we prove equality \eq{tidd}. We have
\begin{align*}
d^2(q_1, q_2) 
&= d^2(\Id, \bq) = 
\left(\sqrt{\frac{\bx \by}{\by-1}} - \sqrt{\frac{\bx}{\by(\by-1}}\right)^2
= \frac{\bx(\by-1)}{\by} = \frac{(x_2-x_1)(y_2-y_1)}{y_1y_2}, \\
\td^2(\tq_1,\tq_2) 
&= (\tx_2-\tx_1)^2 - (\ty_2-\ty_1)^2  = 
\frac 14 \left( -\frac{1}{y_2} + x_2 + \frac{1}{y_1} - x_1\right)^2 -
\frac 14 \left( -\frac{1}{y_2} - x_2 + \frac{1}{y_1} + x_1\right)^2 \\
&= \frac{(x_2-x_1)(y_2-y_1)}{y_1y_2},
\end{align*}
and equality \eq{tidd} follows.
\end{proof}

\begin{remark}
The explicit formulas \eq{**} for the isometry $\map{i}{\Aff}{\Pi}$ were discovered as follows. The exponential mappings for $\Aff$ in the case $K=0$ and for the Minkowski plane $\R^2_1$ have respectively the form:
\begin{align}
&\Exp \ : \ \left(\begin{array}{c} \psi \\ t \end{array}\right)
\mapsto 
\left(\begin{array}{c} x \\ y  \end{array}\right) = 
\left(\begin{array}{c} \rho(f(\tau-\rho) + g\left(\frac{1}{\rho}- \frac{1}{\tau}\right) \\ \frac{\rho}{\tau}  \end{array}\right),\qquad
\widetilde{\Exp} \ : \ \left(\begin{array}{c} \widetilde{\psi} \\ \widetilde{t} \end{array}\right)
\mapsto 
\left(\begin{array}{c} \tx \\ \ty  \end{array}\right) = 
\left(\begin{array}{c} \widetilde{t} \cosh \widetilde{\psi} \\ \widetilde{t} \sinh \widetilde{\psi}  \end{array}\right).
\label{ExpExp}
\end{align}
We set in these formulas $t = \widetilde{t}$, $\psi = \widetilde{\psi}$, and obtain \eq{i}.
\end{remark}

\begin{remark}
In the case $K= 0$ the group $\Aff$ cannot be isometric to the whole Minkowski space $\R^2_1$ since the first is not geodesically complete (see Cor. {\em\ref{cor:geod_compl}}), while the second is.

It would be interesting to construct isometric embeddings of $\Aff$ to $\SS^2_1$ ($\widetilde{\H^2_1}$) in the case $K>0$ (resp. $K <0$). This is more complicated since in this case the formulas analogous to \eq{ExpExp} are more involved.
\end{remark}

\section{Examples}\label{sec:ex}

In this section we present detailed results for the problems $P_1$--$P_3$ defined in Example \ref{ex:P13}.

\subsection{Problem $P_1$}
In this case $K<0$.
The causal future of the point $\Id$ is 
$
J^+ = \exp(\U) =  \{(x, y) \in G \mid x \geq |y-1| \},
$
see Fig. \ref{fig:att_set1}.

\figout{
\onefiglabelsizen
{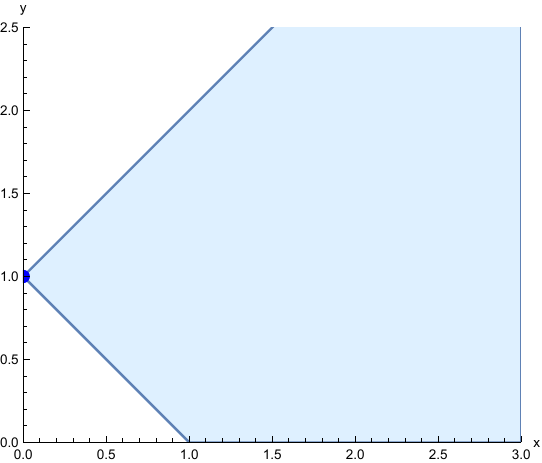}{$J^+$ for the problem $P_1$  }{fig:att_set1}{5}
}

The group $G$ is not globally hyperbolic since for $q_1 = (x_1, y_1) \in G$ with $x_1 > y_1 +1$ the intersection $J^+(\Id) \cap J^-(q_1)$ is not compact, see Fig. \ref{fig:att_set23}. Although, the domain $\intt D = \{(x, y) \in G \mid x > |y-1|, \ x < y + 1\}$ is globally hyperbolic, see Fig. \ref{fig:att_set24}.

\figout{
\twofiglabelsizeh
{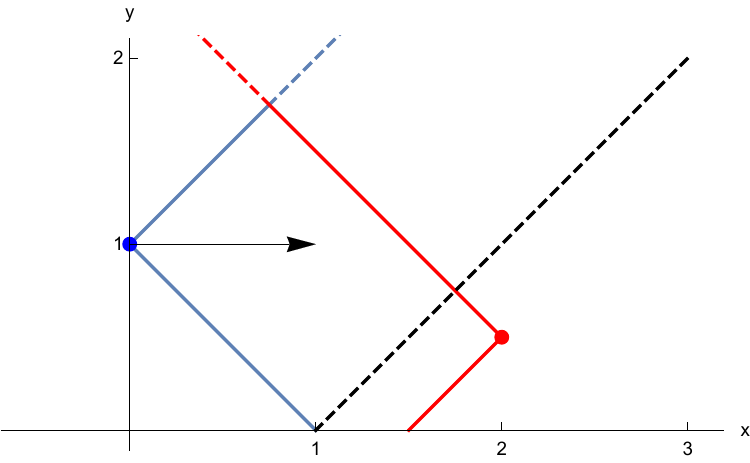}{Problem $P_1$:  $G$ is not globally hyperbolic }{fig:att_set23}{4}
{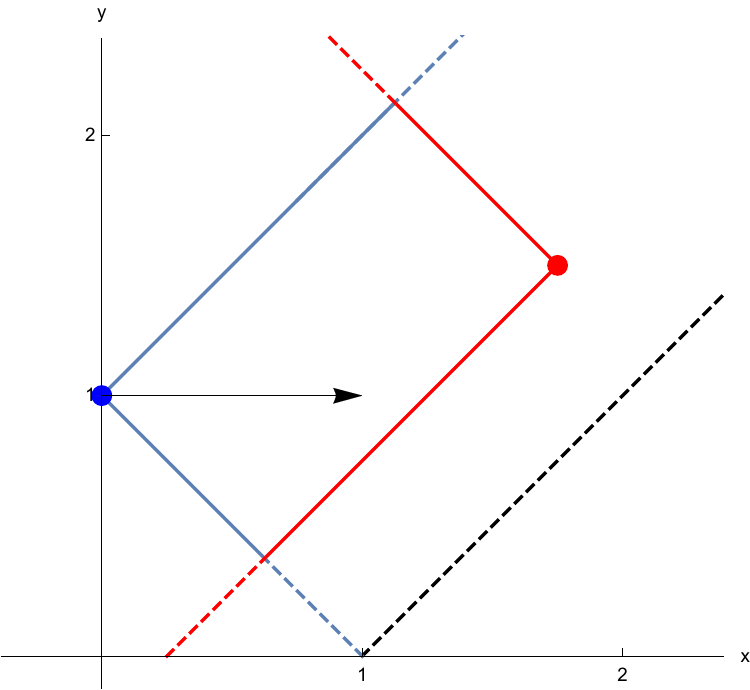}{Problem $P_1$:  $D$ is globally hyperbolic  }{fig:att_set24}{4}
}

\begin{theorem}\label{th:P1}
 Let   $q_1 = (x_1, y_1) \in M \setminus \{\Id\}$ for the problem $P_1$. 
\begin{itemize}
\item[$(1)$]
If $x_1 = |y_1-1|$, then 
$x(t) = \pm(e^{\pm t} - 1)$,
$y(t) = e^{\pm t}$, $\pm = \sgn (y_1 - 1)$, $t_1 = \pm \ln y_1$,
$d(q_1) = 0$.
\item[$(2)$]
If $x_1>|y_1-1|$, then 
\begin{align*}
&x(t) = \cos \rho (\tan \tau - \tan \rho), \qquad y(t) = \frac{\cos \rho}{\cos \tau}, \qquad \tau = \rho + t, \qquad t_1 = \tau-\rho = d(q_1),\\
&\tau = \arcsin \frac{x_1^2+y_1^2 - 1}{2x_1 y_1}, \qquad \rho = \arcsin \frac{y_1^2-x_1^2 - 1}{2x_1},
\end{align*}
the curve $(x(t), y(t))$ is an arc of the hyperbola $y^2 - (x - \sin \rho)^2 = \cos^2 \rho$.  
\end{itemize}
\end{theorem}

\figout{
\twofiglabelsizeh
{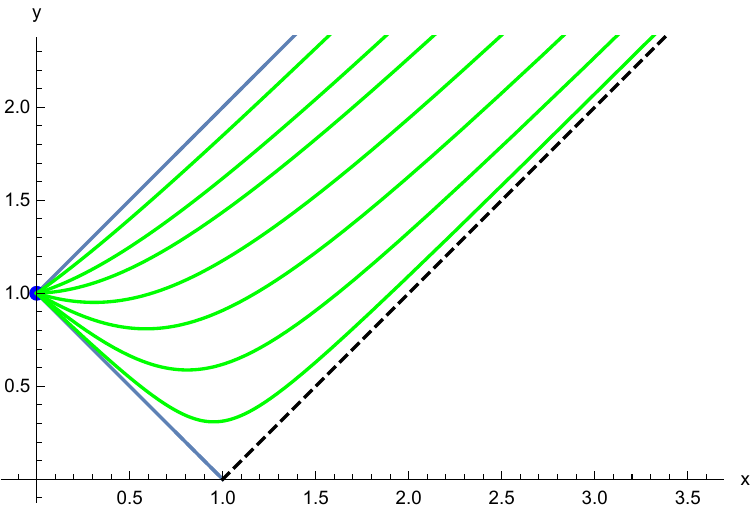}{Lorentzian length maximizers in $P_1$}{fig:extr1}{5}
{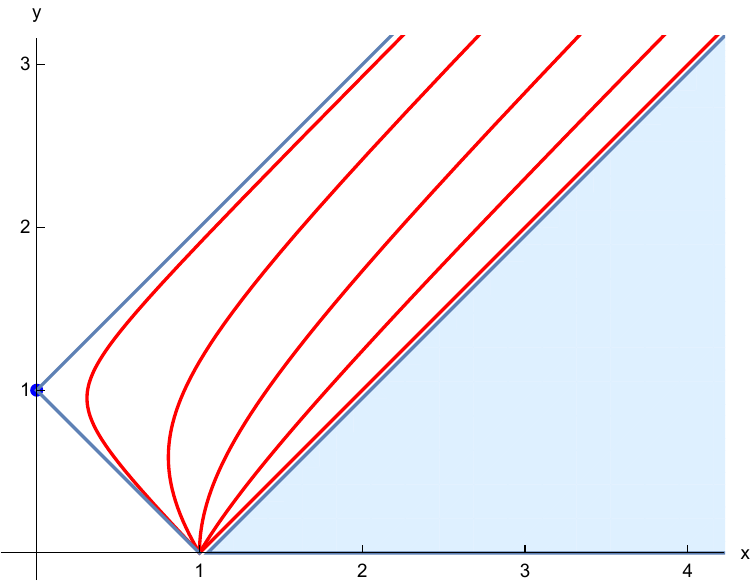}{Lorentzian spheres in $P_1$}{fig:s1}{5}
}

\figout{
\onefiglabelsizen
{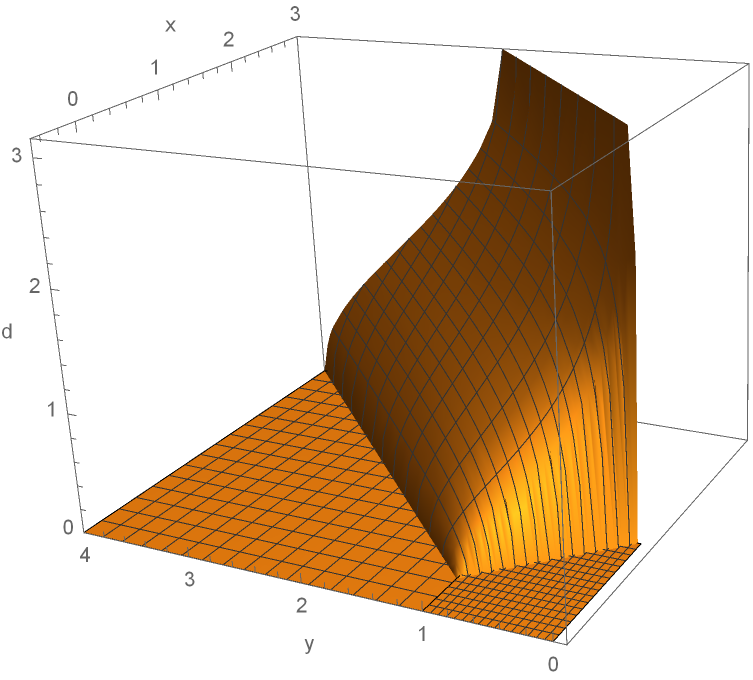}{Plot of Lorentzian distance in  $P_1$}{fig:d1}{7.5}
}

\subsection{Problem $P_2$}
In this case $K>0$. 

\figout{
\twofiglabelsizeh
{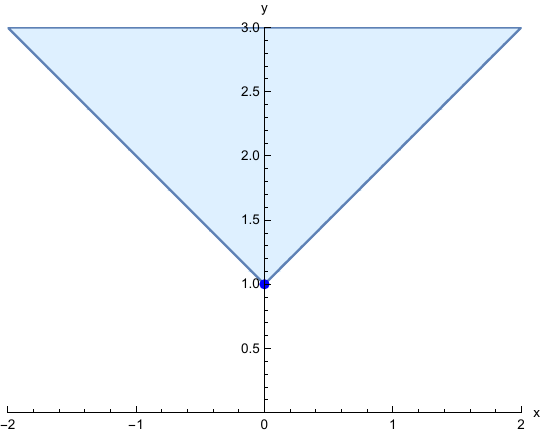}{  $J^+$ for the problem $P_2$   }{fig:att_set2}{5}
{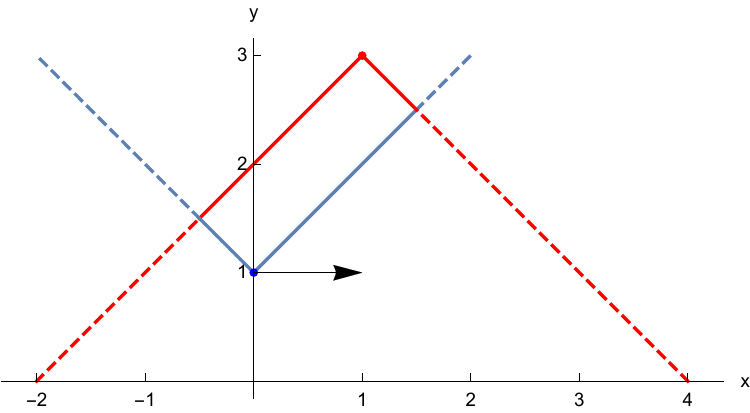}{Problem $P_2$:  $G$ is  globally hyperbolic}{fig:att_set21}{5}
}

\begin{theorem}\label{th:P2}
Let $q_1 = (x_1, y_1) \in J^+ \setminus \{q_0\}$ for the problem $P_2$. 
\begin{itemize}
\item[$(1)$]
If $y_1 - 1 = |x_1|$, then 
$x(t) = \pm (e^t-1)$,
$y(t) = e^t$, $\pm = \sgn x_1$, $t_1 = \ln y_1$,
$d(q_1) = 0$.
\item[$(2)$]
If $x_1 = 0$, then 
$x(t) \equiv 0$,
$
y(t) = e^t$, $t_1 = \ln y_1 = d(q_1)$.
\item[$(3)$]
If $0 < |x_1|<y_1-1$, then 
\begin{align*}
&x(t) = \pm(\sinh \rho \coth \tau - \cosh \rho), \quad y(t) = \frac{\sinh \rho}{\sinh \tau}, \qquad \pm = \sgn x_1,\qquad
\tau = \rho - t, \\
&\rho = \arcosh \frac{1 + x_1^2-y_1^2}{2|x_1|},
\qquad \tau = \arcosh \frac{1 - x_1^2-y_1^2}{2|x_1| y_1}, 
\qquad t_1 = \rho - \tau = d(q_1),
\end{align*}
is the arc of the hyperbola $(\pm x + \cosh \rho)^2 - y^2 = \sinh^2 \rho$.  
\end{itemize}
\end{theorem}

\figout{
\twofiglabelsizeh
{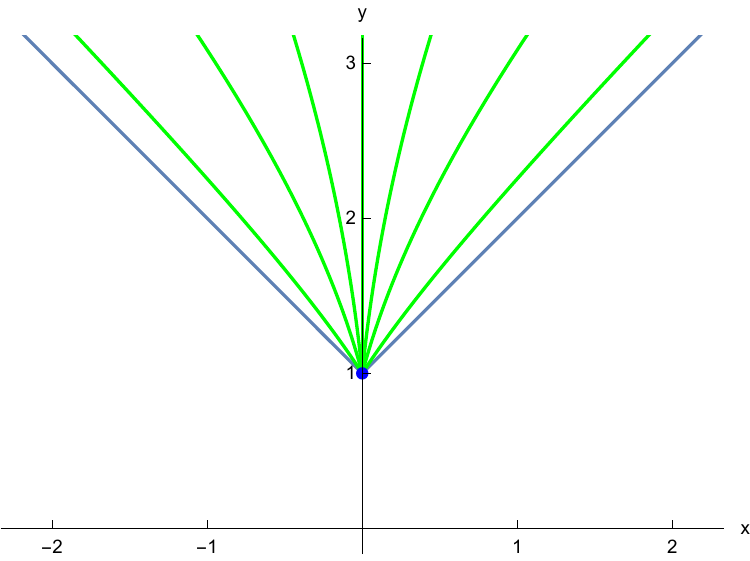}{Lorentzian length maximizers in $P_2$}{fig:extr2}{5}
{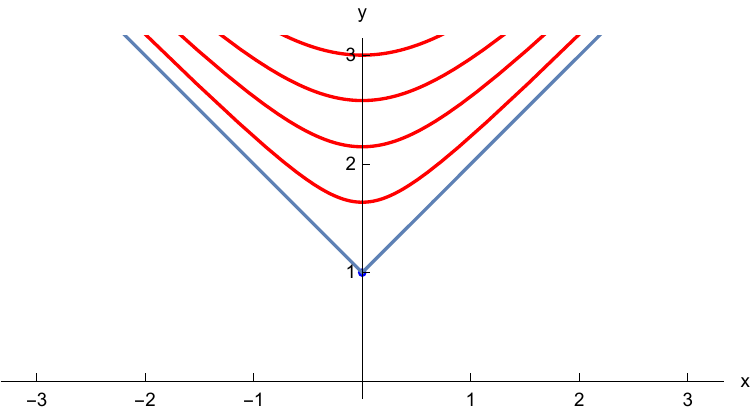}{Lorentzian spheres in $P_2$}{fig:s2}{5}
}

\figout{
\onefiglabelsizen
{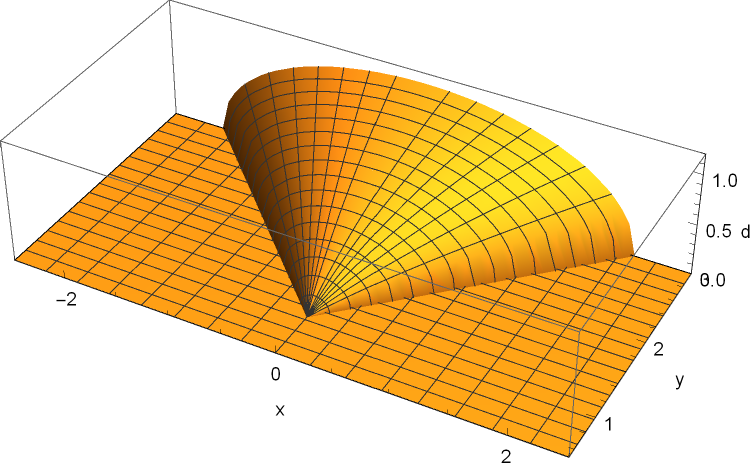}{Plot of Lorentzian distance in  $P_2$}{fig:d2}{7.5}
}

\subsection{Problem $P_3$}
In this case $K = 0$.

\figout{
\twofiglabelsizeh
{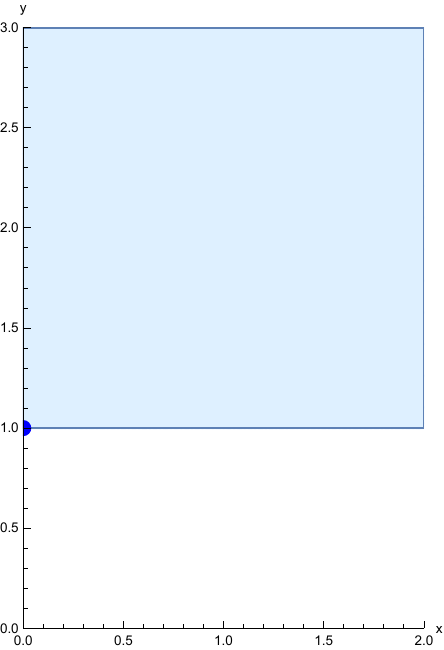}{  $J^+$ for the problem $P_3$  }{fig:att_set3}{6
}
{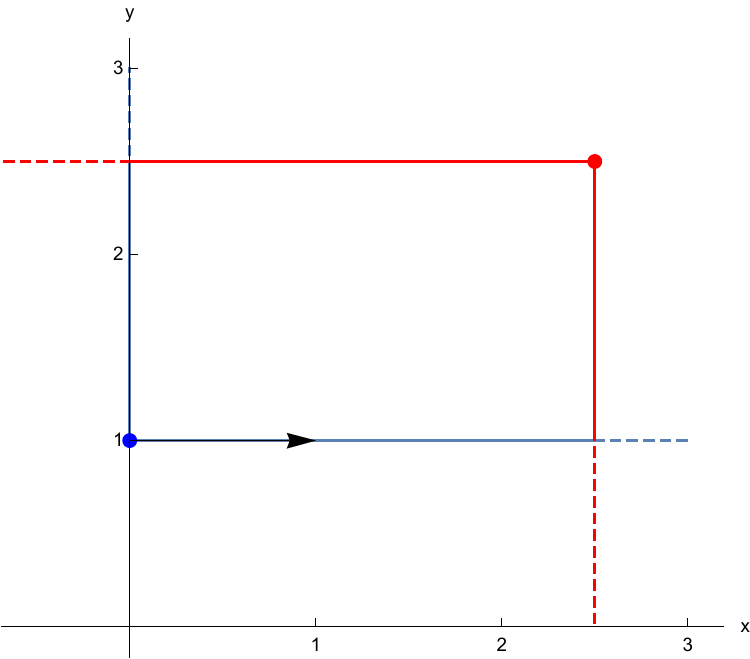}{Problem $P_3$:  $G$ is  globally hyperbolic}{fig:att_set22}{6}
}

\begin{theorem}\label{th:P3}
Let   $q_1 = (x_1, y_1) \in J^+ \setminus \{q_0\}$ for the problem $P_3$. 
\begin{itemize}
\item[$(1)$]
If $x_1 = 0$, then 
$x(t) \equiv 0$,
$y(t) = e^t$, $t_1 = \ln y_1$,
$d(q_1) = 0$.
\item[$(2)$]
If $y_1 = 1$, then 
$x(t) =t$,
$
y(t) \equiv 1$, $t_1 = x_1$,
$d(q_1) = 0$.
\item[$(3)$]
If $x_1 > 0$ and $y_1 > 1$, then $x(t) = \rho( \rho - \tau )$, $y(t) = \frac{\rho}{\tau}$, 
\begin{align*}
&\tau = \rho - t, \qquad \rho = \sqrt{\frac{x_1y_1}{y_1-1}}, 
\qquad \tau = \sqrt{\frac{x_1}{y_1(y_1-1)}}, 
\qquad t_1 = \rho - \tau = d(q_1),
\end{align*}
is the arc of the hyperbola $ x = \rho^2 \left(1 - \frac 1y \right)$.  
\end{itemize}
\end{theorem}

\figout{
\twofiglabelsizeh
{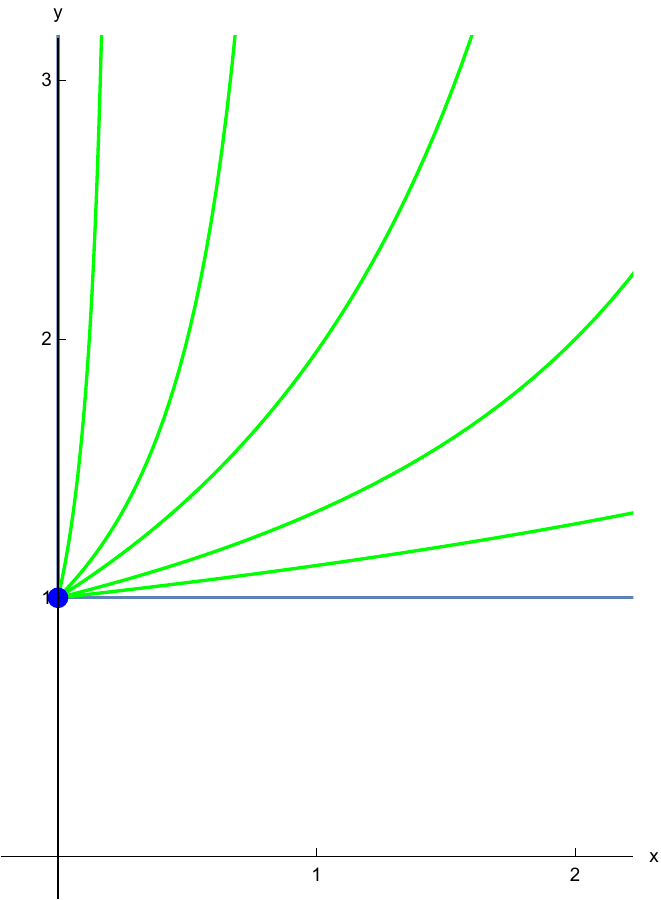}{Lorentzian length maximizers in $P_3$}{fig:extr3}{5}
{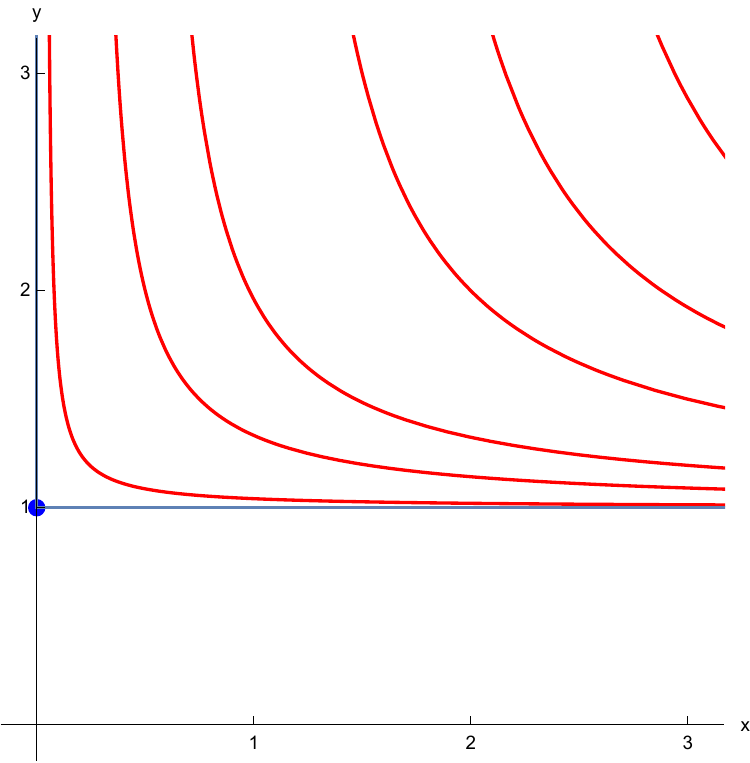}{Lorentzian spheres in $P_3$}{fig:s3}{5}
}

\figout{
\onefiglabelsizen
{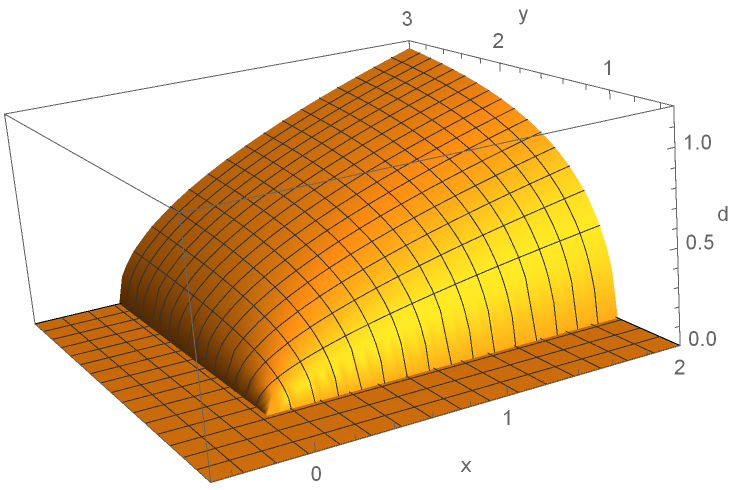}{Plot of Lorentzian distance in  $P_3$}{fig:d3}{7.5}
}

\section*{Acknowledgements}
The author is grateful to L.V. Lokutsievskiy, E. Le Donne, D.V. Alexeevsky, and N.I. Zhukova for discussions of the subject of this work.

\listoffigures


\end{document}